\documentclass[reqno,letterpaper,12pt]{amsart}

\usepackage[]{latexsym,amssymb,amsmath,amsfonts, amsthm}
\usepackage[all,cmtip,ps]{xy}
\usepackage{enumitem}
\usepackage{color}

\def\Z{\mathbb{Z}}

\def\N{{\mathbb N}}

\def\End{\operatorname{End}}
\def\Im{\operatorname{Im}}
\def\Hom{\operatorname{Hom}}
\def\Ext{\operatorname{Ext}}

\def\dualita#1#2{\mathrel{
                 \mathop{\vcenter{
                 \offinterlineskip
                 \hbox to 0.6truecm{\rightarrowfill}%\kern ...
                 \hbox to 0.6truecm{\leftarrowfill}}}%
                 \limits_{#2}^{#1}}}

\DeclareMathOperator{\Ann}{Ann}

\DeclareMathOperator{\Ker}{Ker}

\def\Muni#1#2#3{M_{#2,#1,#3}}
\def\Uni#1#2{U_{#1,#2}}

\newtheorem{theorem}{Theorem}[section]

\newtheorem{corollary}[theorem]{Corollary}

\newtheorem{definition}[theorem]{Definition}
\newtheorem{example}[theorem]{Example}

\newtheorem{lemma}[theorem]{Lemma}

\newtheorem{proposition}[theorem]{Proposition}
\newtheorem{remark}[theorem]{Remark}

\newcommand*{\lMod}{\textrm{\textup{-Mod}}}

\def\Ext{\operatorname{Ext}}

\begin{document}

\title[Pr\"ufer modules over Leavitt path algebras]{Pr\"ufer modules over Leavitt path algebras}

\author{Gene Abrams}
\address{Department of Mathematics, University of Colorado, 1420 Austin Bluffs Parkway, 
Colorado Springs, CO 80918 U.S.A.}
%\curraddr{Department of Mathematics and Statistics,
%Case Western Reserve University, Cleveland, Ohio 43403}
\email{abrams@math.uccs.edu}
%.  \ \  $^*$corresponding author \ \ abrams@math.uccs.edu \ \ 7192553182 
%and its faculty, for its warm hospitality and support by the Visiting Scientist 2012-13 grant of the Universit\`{a} degli Studi di Padova.

\author{Francesca Mantese}
\address{Dipartimento di Informatica, Universit\`{a} degli Studi di Verona,  strada le grazie 15, 37134  Verona, Italy}
\email{francesca.mantese@univr.it}

\author{Alberto Tonolo}
\address{Dipartimento di Matematica Tullio Levi-Civita, Universit\`{a} degli Studi di Padova, via Trieste 63,  35121, Padova, Italy}
\email{tonolo@math.unipd.it}

%    General info
%\subjclass{16S99}
%\date{}


\begin{abstract}
Let $L_K(E)$ denote the Leavitt path algebra associated to the finite graph $E$ and field $K$.  For any closed path $c$  in $E$, we define and investigate the uniserial, artinian, non-noetherian left $L_K(E)$-module $U_{E,c-1}$.   The unique simple factor of each  proper submodule of $U_{E,c-1}$ is  isomorphic to the Chen simple module $V_{[c^\infty]}$.   In our main result, we classify those closed paths $c$ for which $U_{E,c-1}$ is injective.  In this situation, $U_{E,c-1}$ is the injective hull of $V_{[c^\infty]}$. 

\smallskip



\ \ \  corresponding author:  Abrams 

\end{abstract}

\maketitle

\section{Introduction}

Leavitt path algebras have a well-studied,  extremely tight relationship with their {\it pro}jective  modules.   %Given a finite directed graph $E$, one constructs the graph monoid $M_E = (\mathbb{Z}^+)^{|E^0|} / \mathcal{R}$ as the free abelian monoid on the vertex set $E^0$, modulo relations $\mathcal{R}$ which are determined by the edge set $E^1$.  Let $d$ denote $\sum_{v\in E^0}v$ in $M_E$.   Then by \cite[Theorem 3.5]{AMP} $L_K(E)$ is the universal $K$-algebra for which $(\mathcal{V}(L_K(E)), [1_{L_K(E)}]) \cong (M_E, d)$.  Less formally, $L_K(E)$ is the universal (a.k.a. 'Bergman') algebra for which the monoid of finitely generated projective modules naturally gives the graph monoid $M_E$.  
On the other hand, very little is heretofore known about the structure of the {\it in}jective modules over $L_K(E)$.   While the self-injective Leavitt path algebras have been identified in \cite{ARS}, we know of no study of the structure of injective modules over Leavitt path algebras (other than those arising as left ideals).   

We initiate such a study in this article.     For each closed path $c$ in $E$ we construct the  {\it Pr\"ufer module} $\Uni E{c-1}$, recalling the classical construction of Pr\"ufer abelian groups. \textcolor{black}{These modules $\Uni E{c-1}$ are Pr\"ufer also in the sense of Ringel \cite{R6};  indeed, they admit a surjective locally nilpotent endomorphism (see Remark~\ref{rem:quotientofU}).}  In our main result (Theorem \ref{thm:injective}),   we give necessary and sufficient conditions for the injectivity of $\Uni E{c-1}$.  In this case,  $\Uni E{c-1}$ is precisely the injective hull of the Chen simple module $V_{[c^\infty]}$. \textcolor{black}{Our construction is similar to that established by Matlis \cite{Mat} for modules over various commutative noetherian rings, but in a highly noncommutative, non-noetherian setting. }   

Perhaps surprisingly, achieving Theorem \ref{thm:injective} relies on a set of highly nontrivial tools, including:    some general results about uniserial modules over arbitrary associative unital rings;  an explicit description of a projective resolution for $V_{[c^\infty]}$;  a Division Algorithm in $L_K(E)$ with respect to the element $c-1$;  the fact that every Leavitt path algebra is B\'{e}zout (i.e., that every finitely generated one-sided ideal is principal);  and two types of  Morita equivalences for Leavitt path algebras (one of which relates each graph having a source cycle to  a graph having a source loop, the other of which eliminates source vertices).   

The article is organized as follows.  In Section \ref{sec:generalring} we construct what we call ``Pr\"ufer-like modules" over arbitrary unital rings.   In Section \ref{section:Chen} we remind the reader of the construction of the Leavitt path algebra $L_K(E)$ for a directed graph $E$ and field $K$, and describe the Chen simple $L_K(E)$-module $V_{[p]}$ corresponding to an infinite path $p$ arising from $E$.   Specifically, if $c$ is a closed path in $E$, we may build the Chen simple module $V_{[c^\infty]}$.   Continuing our focus on closed paths $c$ in $E$,   in Section \ref{divisionalgorithmsection} we describe a Division Algorithm for arbitrary  elements of $L_K(E)$ by the specific element $c - 1$.   With the discussion from these three sections in hand, we are then in position in Section \ref{Umodulessection} to construct the Pr\"ufer-like   $L_K(E)$-module $U_{E,c-1}$ corresponding to $c-1$.  This sets the stage for our aforementioned  main result (Theorem \ref{thm:injective}), which we present in Section \ref{InjectivePrufersection}.   While one direction of the proof of Theorem \ref{thm:injective} is not difficult, establishing the converse is a much heavier lift; we complete the proof in Sections \ref{Sectiongeneraltosourcecycle} and \ref{Sectionsourcecycletosourceloop}.   Along the way, we will establish in Section \ref{Sectionsourcecycletosourceloop} that the endomorphism ring of $U_{E,c-1}$ is isomorphic to the ring $K[[c-1]]$ of formal power series in $c-1$ with coefficients in $K$, \textcolor{black}{exactly as the ring of $p$-adic integers is isomorphic to the endomorphism ring of $\mathbb Z(p^\infty)$}. 

Unless otherwise stated, all modules are left modules.    
$\N$ denotes the set $\{ 0, 1, 2, \dots\}$.

 \section{Pr\"ufer-like modules}
 \label{sec:generalring}
 
In this section we develop a general ring-theoretic framework for the well-known Pr\"ufer abelian groups $\Z(p^\infty)$.   This framework will provide us with the appropriate context in which to construct the $L_K(E)$-modules $U_{E,c-1}$.  

 Let $R$ be an associative ring with $1\not=0$ and $a\in R$. 
{\it  For the remainder of the section we assume that $a$ is not a right zero divisor}  (i.e., that right multiplication $\rho_a:R\to R$ via $r\mapsto ra$ is a monomorphism of left $R$-modules), {\it and that $a$ is not left invertible} (i.e., that $Ra \neq R$).  
 For each integer $n \in \N_{\geq 1}$ we define the left $R$-module
 $$\Muni nRa:=R/Ra^n,$$ 
 and we denote by  $\eta_{n,a}$  the canonical projection $R\to \Muni nRa$.
By the standing assumptions on $a$, each $\Muni nRa$ is a nonzero  cyclic left $R$-module generated by $1+Ra^n$.  
Moreover, for each $1\leq i<\ell$ we have the following monomorphism of left $R$-modules
\[\psi_{R,i,\ell}:\Muni {i}Ra\to \Muni {\ell}Ra,\quad \mbox{via} \quad  1+Ra^i\mapsto a^{\ell-i}+Ra^{\ell}.\]
\textcolor{black}{The cokernel of $\psi_{R,i,\ell}$ is equal to $\Muni \ell Ra/R(a^{\ell-i}+Ra^\ell)= (R/Ra^\ell)/(Ra^{\ell-i}/Ra^\ell)\cong \Muni {\ell-i} Ra$.}

The left $R$-modules $\Muni nRa$ can be recursively characterized in a categorical way.  

\begin{proposition}
For each $n\geq 2$ the following diagram of left $R$-modules is a pushout.  
%$R$-module $\Muni nRa$ is isomorphic to the pushout $P$ of the diagram
\[\xymatrix{R\ar[d]_{\eta_{n-1,a}}\ar[r]^{\rho_a}&R\ar[d]^{\eta_{n,a}}
%\ar[r]&M_1\ar@{=}[d]\ar[r]&0
\\
\Muni {n-1}Ra \ar[r]^{\psi_{R,n-1,n}}&\Muni {n}Ra.}\]
\end{proposition}
\begin{proof}
Clearly we have $\eta_{n,a}\circ\rho_a=\psi_{R,n-1,n}\circ\eta_{n-1,a}$. 
Let $f:R\to X$ and $g:\Muni {n-1}Ra\to X$ be two homomorphisms of left $R$-modules, for which $g\circ\eta_{n-1,a}=f\circ\rho_a$. It is easy to check that setting $h(1+Ra^n)=f(1)$ defines a left $R$-homomorphism  $h:\Muni {n}Ra\to X$ such that $h\circ \eta_{n,a}=f$ and $h\circ\psi_{R,n-1,n}=g$.
\end{proof}

%\begin{proposition}
%For each $n\geq 2$ the left $R$-module $\Muni nRa$ is isomorphic to the pushout $P$ of the diagram
%\[\xymatrix{R\ar[d]_{\eta_{n-1,a}}\ar[r]^{\rho_a}&R\ar@{.>}[d]^\pi
%%\ar[r]&M_1\ar@{=}[d]\ar[r]&0
%\\
%\Muni {n-1}Ra \ar@{.>}[r]&P.}\]
%\end{proposition}
%\begin{proof}
%The pushout $P$ of the maps $\eta_{n-1,a}$ and $\rho_a$ is isomorphic to the following quotient of the direct product of $\Muni {n-1}Ra$ and $R$:
%\[P\cong\frac{\Muni {n-1}Ra\times R}{\{(\eta_{n-1,a}(r), \rho_a(r)): r\in R\}}=
%\frac{\Muni {n-1}Ra\times R}{\{(r+Ra^{n-1}, ra): r\in R\}}=\frac{\Muni {n-1}Ra\times R}{\langle(1+Ra^{n-1}, a)\rangle}.\]
%Since $\eta_{n-1,a}$ is surjective, so too is  $\pi$. Since
%\[\Ker\pi=\{r\in R: (0,r)\in \langle(1+Ra^{n-1}, a)\rangle\}=Ra^{n},
%\]
%we get $P\cong R/\Ker\pi=R/Ra^{n}=\Muni {n}Ra$.
%\end{proof}

For any $1\leq i<\ell$, using the monomorphism  $\psi_{R,i,\ell}$ allows us to  identify $\Muni {i}Ra$ with its image submodule inside $\Muni {\ell}Ra$.

%Then $\theta_{\ell,i}=\overline{\theta_{\ell,i}}\circ \pi_{i,\ell}$
%where $\pi_{i,\ell}$ is the projection $M_\ell\to M_\ell/(\iota_{i,\ell}(M_i))$ and $\overline{\theta_{\ell,i}}$ is the induced isomorphism.

\begin{proposition}\label{prop:uniserialchar}
Suppose $a\in R$ has these two properties: 

(1)  $\Muni {1}Ra$ is a simple left $R$-module, and

(2)  the equation $a\mathbb{X}=1+Ra^i$ has no solution in $\Muni {i}Ra$ for each $1\leq i< n$.

\noindent
Then the left $R$-module $\Muni {n}Ra$ is uniserial of length $n$. Specifically,  
$\Muni {n}Ra$ has the unique composition series
\[0<\Im\psi_{R,1,n}<\cdots<\Im\psi_{R,n-1,n}<\Muni {n}Ra \ , \]
with all the composition factors isomorphic to $\Muni {1}Ra$.
\end{proposition}
\begin{proof}
By induction on  $n$.\\
Let $n=1$.  By hypothesis, $\Muni {1}Ra$ is simple and hence is uniserial of length 1, and the only composition series is
\[0<\Muni {1}Ra.\]
Now assume that $n>1$. By induction,  $\Muni {1}Ra \ , \dots , \ \Muni {n-1}Ra$ are uniserial,
\[0<\Im\psi_{R,1,n-1}<\cdots<\Im\psi_{R,n-2,n-1}<\Muni {n-1}Ra\]
% for each $1\leq\ell\leq n-1$
is the only composition series of $\Muni {n-1} Ra$, and all  composition factors are isomorphic to $\Muni {1}Ra$. 
For clarity, in the sequel we denote by $H_i$ the submodule $\Im\psi_{R,i,n}$ of  $\Muni {n}Ra$ for each $1\leq i<n$. 
Since $H_i\cong \Im\psi_{R,i,n-1}$ for each $1\leq i<n-1$ and $H_{n-1}\cong \Muni {n-1} Ra$, then
\[0< H_1< H_2< \cdots< H_{n-1}\]
is the unique composition series of $H_{n-1}$, and all the composition factors are isomorphic to $\Muni {1}Ra$. To conclude the proof, we show that if $0\not=L$ is a submodule of $\Muni {n}Ra$, then either $L=\Muni {n}Ra$, or otherwise $L\leq H_{n-1}$, so that  $L=H_i$ for a suitable $1\leq i\leq n-1$.
%
%
%
%Now
%\[0=H_0< H_1< H_2< \cdots< H_{n-1}< \Muni {n} Ra\]
%is a composition series of $\Muni {n} Ra$ with all the composition factors isomorphic to $\Muni {1}Ra$; let us prove 
%
%
%
%Clearly $0=H_0< H_1< H_2< \cdots< H_{n-1}$ are uniserial submodules of $\Muni {n}Ra$.
%We prove that if $L$ is a submodule of $\Muni {n}Ra$, then either $L=\Muni {n}Ra$ or $L\leq H_{n-1}$; then the uniseriality of $\Muni {n}Ra$ follows. Moreover, since $\Muni {n}Ra/H_{n-1}\cong \Muni {1}Ra$ by a previous remark, we will deduce also that all the composition factors are isomorphic to $\Muni {1}Ra$. 
Assume on the contrary that both $L\not= \Muni {n}Ra$ and $L\not\leq H_{n-1}$. Since  then $H_{n-1}\lvertneqq H_{n-1}+L$, and the quotient $\Muni {n}Ra/H_{n-1}\cong \Muni {1}Ra$ is simple, we have $H_{n-1}+L=\Muni {n}Ra$ and $H_{n-1}$ is not contained in $L$. Therefore 
\[\Muni {n}Ra/(L\cap H_{n-1})=(H_{n-1}+L)/(L\cap H_{n-1})=H_{n-1}/(L\cap H_{n-1})\oplus L/(L\cap H_{n-1}).\]
%
%L/(L\cap \psi_{R,n-1,n}\big(M_{n-1,R}(a)\big))\oplus \psi_{R,n-1,n}\big(M_{n-1,R}(a)\big)/(L\cap \psi_{R,n-1,n}\big(M_{n-1,R}(a)\big)).\]
The left $R$-module $L\cap H_{n-1}$ is properly contained in $H_{n-1}$ and hence equal to some $H_j$ for a suitable $0\leq j< n-1$. Then
\[\Muni {n-j}Ra\cong
\Muni {n}Ra/H_j =
L/H_j \oplus H_{n-1}/H_j \cong L/H_j \oplus \Muni {n-1-j}Ra.\]
Since the direct summands $L/H_j$ and $\Muni {n-1-j}Ra$ are not zero for each $0\leq j<n-1$, then $\Muni {n-j}Ra$ is not indecomposable and hence not uniserial. Therefore, by the induction hypothesis, necessarily $j=0$ and we get $\Muni {n}Ra\cong L\oplus H_{n-1}$ with $L\cong \Muni {1}Ra$.
Consider the  diagram
\[\xymatrix{0\ar[r]
&R\ar[d]_{\eta_{n-1,a}}\ar[rr]^{\rho_a}&&R\ar@{.>}[lld]_\varphi\ar[d]^{\eta_{n,a}}
%\ar[r]&M_1\ar@{=}[d]\ar[r]&0
\\
0\ar[r]&\Muni {n-1}Ra\ar[rr]^{ \  \ \ \ \psi_{R,n-1,n}}&&\Muni {n}Ra\ar[r]&L\cong \Muni {1}Ra\ar[r]&0 \ \ ;}\]
since the last row splits, there exists the dotted arrow $\varphi$ such that $\varphi\circ \rho_a=\eta_{n-1,a}$. Therefore $X = \varphi(1)$ is a solution of the equation $a\mathbb{X}=1+Ra^{n-1}$ in $\Muni {n-1}Ra$, a contradiction to the hypothesis.   
\end{proof}

The maps $\psi_{R,i,j}:\Muni {i}Ra\to \Muni {j}Ra$, $1+Ra^i\mapsto a^{j-i}+Ra^{j}$, $1\leq i\leq j$, define a direct system of monomorphisms $\{\Muni {i}Ra,\psi_{R,i,j}\}_{i\leq j}$. (Here we define $a^0 = 1$.)   

\begin{definition}\label{generalPruferdefinition}
{\rm The {\it $a$-Pr\"ufer module} $\Uni Ra$ is the direct limit
$$\Uni Ra =\varinjlim\{\Muni {i}Ra, \psi_{R,i,j}\}_{i \leq  j}.$$  
We  denote by $\psi_{R,i}: \Muni {i}Ra\to \Uni Ra$, $i\geq 1$,    the induced monomorphisms.}
\end{definition}
\textcolor{black}{Under the assumptions of Proposition~\ref{prop:uniserialchar}
the}, module $\Uni Ra$ is generated by the elements $\alpha_i:=\psi_{R,i}(1+Ra^i)$, $i\geq 1$. Clearly, $\Muni {i}Ra\cong R\alpha_i\leq \Uni Ra$, and 
\[a\alpha_{i}=\begin{cases}
     0 & \text{if }i=1, \\
     \alpha_{i-1} & \text{if }i>1.
\end{cases}\]
\begin{proposition}\label{genUisuniserialandartinian}
\textcolor{black}{If $\Muni nRa$ is uniserial of length $n$ for each $n\geq 1$, then the module $\Uni Ra$ is uniserial  and artinian (and not noetherian).  }
\end{proposition}
\begin{proof}
%Since $1+L_K(E)(c-1)^i\mapsto (c-1)^{j-i}+L_K(E)(c-1)^{j}$ defines a monomorphism $\psi_{i,j}:M_i\to M_{j}$ for each $j\geq i$, then for any $i$ the map $\varinjlim_j\psi_{i,j}$ define a monomorphism $M_i\to U$. We have 
%\[\varinjlim_j\psi_{1,j}(M_1)\leq \varinjlim_j\psi_{2,j}(M_2)\leq \cdots\]
%and their union coincides with $U$. We continue to denote by $M_i$ the image of the map $\varinjlim_j\psi_{i,j}$. 
We show that, if $0<N\leq\Uni Ra$, then either $N=R\alpha_j$ for a suitable $j\in \mathbb{N}_{\geq 1}$, or $N=\Uni Ra$.
If $N$ is finitely generated, since $\Uni Ra=\bigcup_i R\alpha_i$ there exists a minimal integer $j\geq 1$ such that $N\leq R\alpha_j<\Uni Ra$: in particular $\Uni Ra$ is not finitely generated and hence not noetherian. Since, by Proposition~\ref{prop:uniserialchar},  $R\alpha_j$ is uniserial and its non-zero submodules are the $R\alpha_\ell$ for $1\leq \ell\leq j$,  we conclude  $N=R\alpha_j$.\\
If $N$ is not finitely generated, write $N=\varinjlim N_{\lambda}$, where the $N_{\lambda}$ are the finitely generated submodules of $N$. For any $\lambda$, by the previous paragraph, there exists $j_{\lambda}$ such that $N_{\lambda}=R\alpha_{j_{\lambda}}$. Since $N\not=N_\lambda$ for any $\lambda$, the sequence $({j_{\lambda}})_\lambda$ is unbounded, so that $N$ contains $R\alpha_\ell$ for every $\ell \in \N$,   and so $N=\Uni Ra$.\\
Hence $\{R\alpha_i  \ | \ i\in\mathbb{N}_{\geq 1}\}$ is  the lattice of the proper submodules of $\Uni Ra$. It is totally ordered and so $\Uni Ra$ is uniserial. 
%Moreover $U$ has essential socle $S_1$, so it is indecomposable. 
Since any $R\alpha_i$ is of finite length, we conclude that $\Uni Ra$ is artinian.  \end{proof}
%Since $\Uni Ra$ is not finitely generated, it cannot be noetherian.   
%\textcolor{black}{
\begin{remark}\label{rem:quotientofU}
Considering the direct limit of the sequences
\[\xymatrix{0\ar[r]& \Muni {i}Ra\ar[r]^{\psi_{R,i,\ell}}& \Muni {\ell}Ra\ar[r]& \Muni {\ell-i}Ra\ar[r]&0}, \quad \ell\in\mathbb N_{\geq i}\]
we get the short exact sequence
\[\xymatrix{0\ar[r]& \Muni {i}Ra\ar[r]^{\psi_{R,i}}& \Uni Ra\ar[r]^{\phi_{R,i}}& \Uni Ra\ar[r]&0.}\]
Therefore all the proper quotients of $\Uni Ra$ are isomorphic to $\Uni Ra$.
\textcolor{black}{Each $\phi_{R,i}$ is a surjective, locally nilpotent endomorphism with kernel of finite length: therefore $\Uni Ra$ is a Pr\"ufer module also in the sense of Ringel \cite{R6}.}
\end{remark}
%}

\begin{example}\label{ex:abeliangroups}
{\rm If $R=\mathbb Z$ and $a=p$ is a prime number, then
$\Muni i{\mathbb Z}p=\mathbb Z/p^i\mathbb Z$, 
and $\Uni {\mathbb Z}p$ is the standard Pr\"ufer abelian group $\mathbb Z(p^\infty)$.
}
\end{example}

%%%%%%Morita
Let $\varepsilon\in R$ be an idempotent such that $R=R\varepsilon R$. Then \cite[Section 22]{AF} the rings $R$ and $S:=\varepsilon R \varepsilon$ are Morita equivalent;  the Morita equivalence is induced by the functors:
\[\Hom_R(R\varepsilon, -):R\lMod \dualita{}{}S\lMod:R\varepsilon\otimes_S-.\]
It is well known (and easy to verify) that, for each left $R$-module $M$, the map
$\varphi\mapsto \varphi(\varepsilon)$ defines a natural isomorphism between the left $S$-modules $\Hom_R(R\varepsilon, M)$ and $\varepsilon M$.

\begin{proposition}\label{prop:Moritaeq}
Let $\varepsilon\in R$ with $\varepsilon^2=\varepsilon$ and $R = R \varepsilon R$. Set $S=\varepsilon R\varepsilon$. Assume $a\in R$ has these two properties:  
\begin{enumerate}
%\item $a$ is not a right zero divisor;
\item $\varepsilon a=a\varepsilon$, and 
\item $a(1-\varepsilon)=u(1-\varepsilon)$ for some invertible central element $u$ of $R$.
\end{enumerate}
%Then $\varepsilon a$ belongs to $S$, and $Sa^n=S\varepsilon a^n=S(\varepsilon a)^n$ is a left $S$-ideal for each $n\in\mathbb N$. 
%
%
%
%
%
%\bigskip
%
%{\bf A/F:  I don't understand this next sentence;  I think that the notation is not correct (should $n$ actually be $i$?)   We need to make this clearer. }
%
%\bigskip
%
%
%
%{\bf A/F:  I think we should shorten this statement to only include the conclusion ``Therefore the Morita equivalence ..."   The other parts of the statement should be part of the proof.}
%
%
%\bigskip
%Moreover, for each $1\leq i<j$ and $n\in\mathbb N$ the positions
%\[\varepsilon r\varepsilon+Sa^n\mapsto \varepsilon r\varepsilon +\varepsilon Ra^n\quad\text{and}\quad 1_S+S(\varepsilon a)^i\mapsto (\varepsilon a)^{j-i}+S(\varepsilon a)^j\]
%define an isomorphism $\nu_n:\Muni {i}R{\varepsilon a}:=S/S(\varepsilon a)^n\to \Hom(R\varepsilon, \Muni {n}Ra)=\varepsilon \Muni {n}Ra$ and a monomorphism $\psi_{S,i,j}$ giving rise to the following commutative diagram of left $S$-modules:
%\[\xymatrix{
%\Muni {i}S{\varepsilon a}\ar[rr]^{\psi_{S,i,j}}\ar[d]_{\cong}^{\nu_i}&& \Muni {j}S{\varepsilon a}\ar[d]_{\cong}^{\nu_j}\\
%\varepsilon \Muni {i}Ra\ar[rr]^{\Hom_R(R\varepsilon, \psi_{R,i,j})}&&\varepsilon \Muni {j}Ra.
%}\]
Then $\varepsilon a=\varepsilon a\varepsilon$ is neither a right zero divisor nor  left invertible  in $S$.
Moreover, the Morita equivalence between the rings $R$ and $S$ sends the direct system of monomorphisms $\{\Muni {i}Ra, \psi_{R,i,j}\}$ to the direct system
 of monomorphisms $\{\Muni {i}S{\varepsilon a}, \psi_{S,i,j}\}$, and sends the Pr\"ufer module $\Uni Ra$ to the Pr\"ufer module $\Uni S{\varepsilon a}$.
 \end{proposition}

\begin{proof}
By (1), $\varepsilon a=\varepsilon^2 a=\varepsilon a\varepsilon$
belongs to $S$.  If $\varepsilon a$ were a right zero divisor in $S$ there would exist  $r\in R$ such that $0=\varepsilon a\varepsilon r\varepsilon=a\big(\varepsilon r\varepsilon
\big)$, contradicting the standing assumption that $a$ is not a right zero divisor in $R$.   If $\varepsilon a$ were left invertible in $S$, there would exist $r_1\in R$ such that 
 $\varepsilon r_1\varepsilon a=\varepsilon$; then by (1) and (2)
\[1=\varepsilon +(1-\varepsilon)=\varepsilon r_1\varepsilon a+u^{-1}u(1-\varepsilon)=\varepsilon r_1\varepsilon a+u^{-1}a(1-\varepsilon)=\big(\varepsilon r_1\varepsilon+u^{-1}(1-\varepsilon)\big)a , \]
contradicting the standing assumption  that $a$ is not left invertible in $R$. \\
By (1), $Sa^n=S\varepsilon a^n=S(\varepsilon a)^n$ is a left $S$-ideal for each $n\in\mathbb N$.
We have the following commutative diagram with exact rows
\[\xymatrix{0\ar[r]& R a^n\varepsilon=R \varepsilon a^n\ar@{^(->}[d]\ar[r]&R\varepsilon\ar@{^(->}[d]\\
0\ar[r]&R a^n\ar[r]&R\ar[r]& \Muni {n}Ra\ar[r]& 0.
}\]
Applying the functor $\Hom_R(R\varepsilon, -)$ 
we get the following commutative diagram of left $S$-modules with exact rows and columns:
\[\xymatrix{
0\ar[r]&\varepsilon Ra^n\varepsilon=\varepsilon R\varepsilon a^n=Sa^n=S(\varepsilon a)^n\ar@{^(->}[d]\ar[r]&S\ar[r]\ar@{^(->}[d]&\Muni {n}S{\varepsilon a}\ar[d]^{\nu_n}\ar[r]&0\\
0\ar[r]&\varepsilon R a^n\ar@{->>}[d]\ar[r]&\varepsilon R\ar@{->>}[d]\ar[r]& \varepsilon \Muni {n}Ra\ar@{->>}[d]\ar[r]& 0\\
0\ar[r]&\varepsilon R a^n(1-\varepsilon)\ar[r]^\xi&\varepsilon R(1-\varepsilon)\ar[r]& Q\ar[r]& 0
}\]
where $\nu_n$ sends $\varepsilon r\varepsilon+
S(\varepsilon a)^n$ to $\varepsilon r\varepsilon+
R a^n$.
By (2)  $Ra^n(1-\varepsilon)=Ru^n(1-\varepsilon)=R(1-\varepsilon)$; therefore the map $\xi$ is surjective and hence $Q=0$. Therefore $\nu_n$ is an isomorphism and $\varepsilon r\varepsilon+
 R a^n=\varepsilon r+
 R a^n$: indeed
 \[\varepsilon r-\varepsilon r\varepsilon=\varepsilon r(1-\varepsilon)\in R u^n(1-\varepsilon)=R a^n(1-\varepsilon)= R (1-\varepsilon)a^n\subseteq R a^n.\]
We now show that for any $i\leq j$ the following diagram commutes:
\[\xymatrix{
\Muni {i}S{\varepsilon a}\ar[rr]^{\psi_{S,i,j}}\ar[d]_{\cong}^{\nu_i}&& \Muni {j}S{\varepsilon a}\ar[d]_{\cong}^{\nu_j}\\
\varepsilon \Muni {i}Ra\ar[rr]^{\Hom_R(R\varepsilon, \psi_{R,i,j})}&&\varepsilon \Muni {j}Ra.
}\]
We have:
\begin{align*}
\hspace{-1in} \Hom_R(R\varepsilon, \psi_{R,i,j})\big(\nu_i(\varepsilon r\varepsilon+S(\varepsilon a)^i)\big) &=\Hom_R(R\varepsilon, \psi_{R,i,j})(\varepsilon r\varepsilon+Ra^i)\\
&=\varepsilon r\varepsilon a^{j-i}+Ra^j \\
%&=\varepsilon r a^{j-i}\varepsilon+Ra^j \\
&=\varepsilon r \varepsilon (\varepsilon a)^{j-i}+Ra^j \\
&=\nu_j\big(\varepsilon r \varepsilon (\varepsilon a)^{j-i}+S(\varepsilon a)^j \big)\\
&=\nu_j\big(\psi_{S,i,j}(\varepsilon r \varepsilon+S (\varepsilon a)^i\big) .
%\\
%
%
%
%
%
%
%
%&=
%\varepsilon r a^{j-i}\varepsilon+Ra^j \\
%&= \nu_j(\varepsilon ra^{j-i}\varepsilon+Sa^j) \\
%& =\nu_j\big((\varepsilon a)^{j-i}+S(\varepsilon a)^j\big)\\
%& =\nu_j\big(\psi_{S,i,j}(1_S+S(\varepsilon a)^i\big) \\
%&=\nu_j\big(\psi_{S,i,j}(1_S+Sa^i\big)).
\end{align*}
%
%
%
%
%
%By (3)  $Ra^n(1-\varepsilon)=Ru^n(1-\varepsilon)=R(1-\varepsilon)$; therefore the map $\xi$ is surjective and hence $\Hom_R(R\varepsilon,\Muni {n}Ra)\cong\varepsilon \Muni {n}Ra=\Muni {n}S{\varepsilon a}$.   More precisely, for each $r\in R$ one has 
%\[\varepsilon r-\varepsilon r\varepsilon=\varepsilon r(1-\varepsilon)\in \varepsilon R u^n(1-\varepsilon)=\varepsilon R a^n(1-\varepsilon)= \varepsilon R (1-\varepsilon)a^n\subseteq \varepsilon R a^n,\]
%and hence $\varepsilon r\varepsilon+Sa^n\mapsto \varepsilon r+\varepsilon Ra^n$ is an isomorphism.
%Finally we have:
%\begin{align*}
%\hspace{-1in} \Hom_R(R\varepsilon, \psi_{R,i,j})\big(\nu_i(1_S+Sa^i)\big) &=\Hom_R(R\varepsilon, \psi_{R,i,j})(\varepsilon+\varepsilon Ra^i)\\
%&=\varepsilon a^{j-i}+\varepsilon Ra^j \\
%&=
%\varepsilon a^{j-i}\varepsilon+\varepsilon Ra^j \\
%&= \nu_j(\varepsilon a^{j-i}\varepsilon+Sa^j) \\
%& =\nu_j\big((\varepsilon a)^{j-i}+S(\varepsilon a)^j\big)\\
%& =\nu_j\big(\psi_{S,i,j}(1_S+S(\varepsilon a)^i\big) \\
%&=\nu_j\big(\psi_{S,i,j}(1_S+Sa^i\big)).
%\end{align*}
%\[\Hom_R(R\varepsilon, \psi_{R,i,j})\big(\nu_i(1_S+Sa^i)\big)=\Hom_R(R\varepsilon, \psi_{R,i,j})(\varepsilon+\varepsilon Ra^i)=\varepsilon a^{j-i}+\varepsilon Ra^j=
%\varepsilon a^{j-i}\varepsilon+\varepsilon Ra^j=\]
%\[=\nu_j(\varepsilon a^{j-i}\varepsilon+Sa^j)=\nu_j\big((\varepsilon a)^{j-i}+S(\varepsilon a)^j\big)=\nu_j\big(\psi_{S,i,j}(1_S+S(\varepsilon a)^i\big)
%=\nu_j\big(\psi_{S,i,j}(1_S+Sa^i\big)).
%\]
Therefore the Morita equivalence between $R$ and $S$ sends the direct system of monomorphisms $\{\Muni {i}Ra, \psi_{R,i,j}\}$ to the direct system
 of monomorphisms $\{\Muni {i}S{\varepsilon a}, \psi_{S,i,j}\}$. Since Morita equivalences commute with direct limits, we get also that the Pr\"ufer module $\Uni Ra$ is sent to the Pr\"ufer module $\Uni S{\varepsilon a}$.
\end{proof}

%Since $(1-\varepsilon)a^n=a^n-\varepsilon a^n=a^n-u^na^n\varepsilon=a^n(1-u^n\varepsilon)$ one has
%$\varepsilon R (1-\varepsilon)a^

%If $u$ and $u^{-1}$ belong to $S$, then
%\[S(\varepsilon a^n)=S(\varepsilon u^na^n\varepsilon)=Su^n(\varepsilon a^n\varepsilon)=Sa^n\]

 %Since for any $s\in S\subseteq R$ it is $s(\varepsilon a^n)=(s\varepsilon)a^n=sa^n$, the element $a^n\in R$ generates a cyclic left $S$-module isomorphic to the left ideal $S(\varepsilon a^n)$ of $S$. In such a case we have that 
%\[\varepsilon Ra^n\varepsilon=\varepsilon Ra^{n-1}u\varepsilon a=\cdots=\varepsilon R u\varepsilon a^n=\]
%$\Hom_R(R\varepsilon, Ra^n)\cong \varepsilon Ra^n$ contains
% and $a(1-\varepsilon)=u'(1-\varepsilon)$ for suitable invertible elements $u, u'\in R$.

%%%%%%%%%

%%%%%%%%%%%%%%%%%%%%%%%%%%%%%%%%%%%%
%
%
%
%%      \section{A division algorithm in $L_K(E)$}\label{divisionalgorithmsection}
%
%
%
%%%%%%%%%%%%%%%%%%%%%%%%%%%%%%%%%%%%%

\section{Chen simple modules over Leavitt path algebras}\label{section:Chen}

In this section we give  a (minimalist)  review of the germane notation, first about directed graphs, then about Leavitt path algebras, and finally about Chen simple modules. 

    A (directed) graph $E = (E^0, E^1, s,r)$ consists of a {\it vertex set} $E^0$, an {\it edge set} $E^1$, and {\it source} and {\it range} functions $s, r: E^1 \rightarrow E^0$.  For $v\in E^0$, the set of edges $\{ e\in E^1 \ | \ s(e)=v\}$ is denoted $s^{-1}(v)$. $E$ is called {\it finite} in case both $E^0$ and $E^1$ are finite sets.  
%  $E$ is called {\it row-finite} in case $s^{-1}(v)$ is finite for every $v\in E^0$.   
%Unless specifically indicated otherwise, throughout this article we assume that $E$ is a finite graph, i.e., that the sets $E^0$ and $E^1$ are finite.    
A {\it path} $\alpha$ in $E$ is a sequence $e_1 e_2 \cdots e_n$ of edges in $E$ for which $r(e_i) = s(e_{i+1})$ for all $1 \leq i \leq n-1$.  
%We denote by $\alpha^0$ the set $\{s(e_1), \dots, s(e_n)\}$.    
We say that such $\alpha$ has {\it length} $n$, and we write $s(\alpha) = s(e_1)$ and $r(\alpha) = r(e_n)$.  We view each vertex $v \in E^0$ as a path of length $0$, and denote $v = s(v) = r(v)$.  We denote the set of paths in $E$ by ${\rm Path}(E)$.   We say a vertex $v$ {\it connects to} a vertex $w$ in case $v=w$, or there exists a path $\alpha$ in $E$ for which $s(\alpha) = v$ and $r(\alpha) = w$.  
A  path $\textcolor{black}{\gamma} = e_1 e_2 \cdots e_n$ ($n \geq 1$)  in $E$   is {\it closed}  in case $r(e_n) = s(e_1)$.

Unfortunately, the phrase ``simple closed path" has come to be defined as two distinct concepts in the literature.    We choose in the current article to follow what now seems to be the more common usage.  Specifically, for a closed path   $\textcolor{black}{\gamma} = e_1 e_2 \cdots e_n$, we call $\textcolor{black}{\gamma}$    {\it simple} in case $s(e_i) \neq s(e_1)$ for all $1 < i  \leq n$, and we call $\textcolor{black}{\gamma}$ {\it basic} in case \textcolor{black}{$\gamma \neq \delta^k$ for any closed path $\delta$} and positive integer $k$.   (In our previous article \cite{AMT1} we  followed Chen's usage of this phrase given in \cite{Chen}; in those two places, ``simple closed path" means what we are now calling a ``basic closed path".)

Some additional properties of closed paths will play a role in the sequel.  If $\gamma  =e_1e_2\cdots e_n$ is a closed path in $E$, then a path $\gamma^{\prime}$ of the form $e_i e_{i+1} \cdots e_n e_1 \cdots e_{i-1}$ (for any $1\leq i \leq n$ is called a {\it cyclic shift} of $\gamma$.    The closed path $\textcolor{black}{\gamma} =e_1e_2\cdots e_n$ in $E$ is called a  \emph{cycle} if $s(e_i)\not=s(e_j)$ for each $i\not=j$;  a  \emph{loop} if $n=1$;
 a  \emph{maximal cycle} if $\textcolor{black}{\gamma}$ is a cycle, and there are no cycles in $E$ other than cyclic shifts of $\gamma$ 
 % from $\textcolor{black}{\gamma}$ 
 which connect to $s(\textcolor{black}{\gamma})=s(e_1)$; and a   \emph{source cycle} (resp., \emph{source loop}) if $\textcolor{black}{\gamma}$ is a cycle (resp., loop), and  there are no edges $e\not=e_i$ in $E$ such that $r(e)=r(e_i)$, for $1\leq i\leq n$.    Less formally, a source cycle is a cycle for which no vertices in the graph connect to the cycle, other than those vertices which are already in the cycle.

%  In deference to  \cite{Chen} (but not standard in the literature), a closed path $\sigma$ is called {\it simple} in case $\sigma \neq \beta^m$ for any closed path $\beta$ and integer $m\geq 2.$     
%  A {\it sink} in $E$ is a vertex $w \in E^0$ for which the set $s^{-1}(w)$ is empty, while an {\it infinite emitter} in $E$ is a vertex $u \in E^0$ for which the set $s^{-1}(u)$ is infinite.  

\smallskip

For any field $K$ and  graph $E$ the Leavitt path algebra $L_K(E)$ has been the focus of sustained investigation since 2004.  We give here a basic description of $L_K(E)$; for additional information, see  \cite{TheBook}.       Let $K$ be a field, and let $E = (E^0, E^1, s,r)$ be a directed  graph with vertex set $E^0$ and edge set $E^1$.   The {\em Leavitt path $K$-algebra} $L_K(E)$ {\em of $E$ with coefficients in $K$} is  the $K$-algebra generated by a set $\{v\mid v\in E^0\}$, together with a set of symbols $\{e,e^*\mid e\in E^1\}$, which satisfy the following relations:

(V)   \ \ \  \ $vu = \delta_{v,u}v$ for all $v,u\in E^0$, \  
%(i.e., $\{v\mid v\in E^0\}$ is a set of orthogonal idempotents),

  (E1) \ \ \ $s(e)e=er(e)=e$ for all $e\in E^1$,

(E2) \ \ \ $r(e)e^*=e^*s(e)=e^*$ for all $e\in E^1$,

 (CK1) \ $e^*e'=\delta _{e,e'}r(e)$ for all $e,e'\in E^1$, and

(CK2)ส\ \ $v=\sum _{\{ e\in E^1\mid s(e)=v \}}ee^*$ for every   $v\in E^0$ for which $0 < |s^{-1}(v)| < \infty$.

%An alternate description of $L_K(E)$ may be given as follows.  For any graph $E$ let $\widehat{E}$ denote the ``double graph" of $E$, gotten by adding to $E$ an edge $e^*$ in a reversed direction for each edge $e\in E^1$.   Then $L_K(E)$ is the usual path algebra $K\widehat{E}$, modulo the ideal generated by the relations (CK1) and (CK2).    

\smallskip

It is easy to show that $L_K(E)$ is unital if and only if $|E^0|$ is finite; in this case, $1_{L_K(E)} = \sum_{v\in E^0}v$.    Every element of $L_K(E)$ may be written as $\sum_{i=1}^n k_i \alpha_i \beta_i^*$, where $k_i$ is a nonzero element of $K$, and each of the $\alpha_i$ and $\beta_i$ are paths in $E$.  If $\alpha \in {\rm Path}(E)$ then we may view $\alpha \in L_K(E)$, and will often refer to such $\alpha$ as a {\it real path} in $L_K(E)$; analogously, for $\beta = e_1 e_2 \cdots e_n \in {\rm Path}(E)$ we often refer to the element $\beta^* = e_n^* \cdots e_2^* e_1^*$ of $L_K(E)$ as a {\it ghost path} in $L_K(E)$.     

%\bigskip
%
%{\bf Do we need this?}
%The  map $KE \rightarrow L_K(E)$ given by the $K$-linear extension of $\alpha \mapsto \alpha$ (for $\alpha \in {\rm Path}(E)$) 
%$\sum_{i=1}^n k_i \alpha_i \mapsto \sum_{i=1}^n k_i \alpha_i $, 
% is an injection of $K$-algebras by \cite[Corollary 1.5.12]{TheBook}.   
 
% \textcolor{red}{No}

%\medskip

 {\it We assume throughout the article that $E$ is finite.}   In particular, we assume that $L_K(E)$ is unital.  The multiplicative identity of a ring $R$  will be denoted by $1_{R}$, or more simply by $1$ if the context is clear.

The ideas presented in the following few paragraphs come from \cite{Chen};  however, some of the  notation we use here differs from that used in \cite{Chen}, in order to make our presentation more notationally consistent with the general body of literature regarding Leavitt path algebras.

%\bigskip

%{\bf A/F:   We need to establish some consistent notation.   Chen uses 'cyclic' for an infinite path in a way that is not consistent with the use of the word 'cycle' in usual terminology. }

%\bigskip
  
%A {\it cycle} in $E$ is a path $c = e_1 e_2 \cdots e_n$ in $E$ for which $s(e_1) = r(e_n)$, and for which $s(e_i) \neq s(e_j)$ for $1\leq i\neq j \leq n$.   
%(We note that  the terminology used throughout the literature in this regard is not standardized.)   

Let $p$ be an {\it infinite path in} $E$; that is, $p$ is a sequence $ e_1e_2e_3\cdots$, where $e_i \in E^1$ for all $i\in \N$, and for which $s(e_{i+1}) = r(e_i)$ for all $i\in \N$.   We emphasize that while the phrase {\it infinite path in} $E$ might seem to suggest otherwise, an infinite path in $E$ is not an element of ${\rm Path}(E)$, nor may it be interpreted  as   an element of the path algebra $KE$ nor  of the Leavitt path algebra $L_K(E)$.  (Such a path is sometimes called a {\it left}-infinite path in the literature.)  We denote the set of infinite paths in $E$ by $E^\infty$.   

 Let $c$ be a closed path in $E$.  Then the path $c c c \cdots$ is an infinite path in $E$, which we denote by $c^\infty$; we call such a 
 %.   We call an infinite path of the form $c^\infty$ a 
 {\it cyclic infinite} path.   For $c$ a closed path in $E$ let $d$ be the 
 %simple
 basic closed path in $E$ for which $c = d^n$.   Then $c^\infty = d^\infty$ as elements of $E^\infty$.

For $p = e_1e_2e_3\cdots \in E^\infty$  
 %and $n\in \N$ we denote by $\tau_{\leq n}(p)$, or often more efficiently by $p_n$, the (finite) path $e_1e_2\cdots e_n$, while 
 we denote by $\tau_{>n}(p)$ the infinite path $e_{n+1}e_{n+2}\cdots$.   
 %We note that $\tau_{\leq n}(p)$ is an element of ${\rm Path}(E)$ (and thus may be viewed as an element of $L_K(E)$), and that $p$ is the concatenation $p = \tau_{\leq n}(p) \cdot \tau_{>n}(p)$.  
If $p$ and $q$ are infinite paths in $E$, we say that $p$ and $q$ are {\it tail equivalent} (written $p \sim q$) in case there exist integers $m,n$ for which $\tau_{>m}(p) = \tau_{>n}(q)$; intuitively, $p \sim q$ in case $p$ and $q$ eventually become the same infinite path.   Clearly $\sim$ is an equivalence relation on $E^\infty$, and we let $[p]$ denote the $\sim$ equivalence class of the infinite path $p$.

The infinite path $p$ is called {\it rational} in 
%case there exists $n\in \N$ for which $p \sim \tau_{>n}(p)$; intuitively, in case $p$ is tail equivalent a proper tail of itself.   
case $p \sim c^\infty$ for some closed path $c$.   By a previous observation, we may assume without loss of generality that such $c$ is a 
%simple 
basic closed path.      In particular, for  any  closed path $c$ we have that $c^\infty$ is rational.   

% If $p \in E^\infty$ is not rational we say $p$ is {\it irrational}.   

%\begin{example}\label{R2Example}
%{\rm 
%Let $R_2$ denote the ``rose with two petals" graph
%$$\xymatrix{\bullet^v \ar@(ul,dl)_e \ar@(ur,dr)^f} \ .$$
%It is not hard to show that $p$ is rational precisely when $p \sim \sigma^\infty$, where $\sigma$ is a simple closed path.  
%(See, e.g., \cite[p. 4]{C}.)   
%\noindent
%Then $q = efeffefffeffffe\cdots$ is an irrational infinite path in $R_2^\infty$.  Indeed, it is easy to show that there are uncountably many distinct irrational infinite paths in $R_2^\infty$.   We note additionally that there are infinitely many simple closed paths in ${\rm Path}(R_2)$, for instance, any path of the form $ef^i$ for $i\in \N_{\geq 1}$.  
%}
%\end{example}

%{\bf Do we want to use the $\hat{\rho}$ notation, or just use $\rho$?} 
%\textcolor{red}{We can use just $\rho$}
 Let $M$ be a left $L_K(E)$-module.  For each $m\in M$ we define the $L_K(E)$-homomorphism ${\rho}_m: L_K(E) \rightarrow M$, given by ${\rho}_m(r) = rm$.   The restriction of the right-multiplication map ${\rho}_m$ may also be viewed as an $L_K(E)$-homomorphism from any left ideal $I$ of $L_K(E)$ into $M$.  
%When $I = L_K(E)v$ for some vertex $v$ of $E$, we will denote $\hat{\rho}_m$ simply by $\rho_m$. 

\medskip
 
Following \cite{Chen}, for any infinite path $p$ in $E$  we construct a simple left $L_K(E)$-module $V_{[p]}$, as follows.   

\begin{definition} \label{Chendef}
{\rm 
Let $p$ be an infinite path in the graph $E$, and let $K$ be any field.  Let $V_{[p]}$ denote the $K$-vector space having basis $[p]$, that is, having basis consisting of distinct elements of $E^\infty$ which are tail-equivalent to $p$.    For any $v\in E^0$, $e\in E^1$, and $q = f_1f_2f_3\cdots  \in [p]$, define 
$$
 v \cdot q =
 \begin{cases}
q  &\text{if }  v = s(f_1) \\
0 &\text{otherwise,} 
\end{cases}
    \ \ \ 
 e\cdot q=
  \begin{cases}
eq  &\text{if }  r(e) = s(f_1) \\
0 &\text{otherwise,} 
\end{cases}
\ \ \ \ \mbox{and} \qquad 
e^* \cdot q = 
 \begin{cases}
\tau_{>1}(q)  &\text{if }  e = f_1 \\
0 &\text{otherwise.} 
\end{cases}
$$
\noindent
Then  the $K$-linear extension to all of $V_{[p]}$ of this action endows $V_{[p]}$ with the structure of a left $L_K(E)$-module.   
}
\end{definition}

\begin{theorem} \label{Chentheoremforsimples}  (\cite[Theorem 3.3]{Chen}).  Let $E$ be any directed graph and $K$ any field.  Let $p\in E^\infty$.   Then the left $L_K(E)$-module $V_{[p]}$ described in  Definition \ref{Chendef}
%\ref{DefintionofChensimpleforinfinitepaths}
 is simple.   Moreover, if $p,q \in E^\infty$, then $V_{[p]} \cong V_{[q]}$ as left $L_K(E)$-modules if and only if $p \sim q$, which happens precisely when $V_{[p]} = V_{[q]}$.
\end{theorem}

We will refer to a module of the form $V_{[p]}$ as presented in Theorem \ref{Chentheoremforsimples} as a {\it Chen simple module}.

Because  $V_{[c^\infty]}=V_{[{(c^2)}^\infty]}$ for any closed path $c$ in $E$,  when analyzing  Chen simple modules $V_{[c^\infty]}$ we can without loss of generality assume that $c$ is a   basic closed path. Observe that if $c=e_1\cdots e_n$ and $d$ are two basic closed paths, then $[c^\infty]=[d^\infty]$ if and only if $d=e_ie_{i+1}\cdots e_ne_1\cdots e_{i-1}$ for a suitable $1\leq i\leq n$.

\begin{example}
Let $E=R_2$ be the rose with two petals:
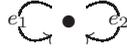
\[\xymatrix{\bullet\ar@(ul,dl)[]|{e_1}\ar@(ur,dr)[]|{e_2}}\]
%{\bf We need to look at closed simple paths here, rather than irrational paths.} 
\textcolor{black}{Then, for example, the infinite paths $p=e_1e_2^2e_1e_2^2e_1e_2^2\cdots$ and $q=e_1e_2e_1e_2e_1e_2\cdots$ are rational paths which are not tail equivalent.}
\end{example}

For the sake of completeness and reader convenience, we state and briefly sketch proofs of  the following two lemmas.  These  include, in the case of a finite graph, some slight generalization of the results obtained in \cite[Lemma 2.5, Proposition 2.6, Lemma 2.7, Theorem 2.8]{AMT1}.

%For sake of completeness, we prove here the following two lemmas which barely generalise \cite[Lemma~2.5, Proposition 2.6, Lemma~2.7, Theorem 2.8]{AMT1}. 
\begin{lemma}\label{lemma:riassunto1}
Let $E$ be a finite graph.

(1)  Let  $c$ be a closed path in $E$, and $r\in L_K(E)$. Then $r(c-1) =0$ in $L_K(E)$ if and only if $r=0$.

(2)  Let  $c$ be a basic closed path in $E$.   Let $\alpha, \beta \in {\rm Path}(E)$ for which $0 \neq \alpha c^\infty = \beta c^\infty$ in $V_{[c^\infty]}$.   Suppose also that $\alpha \neq \gamma c^N$ and $\beta \neq \delta c^{N'}$ for any $\gamma, \delta \in {\rm Path}(E)$ and positive integers $N,N'$.     Then $\alpha = \beta$.  

(3)  Let $c$ be a basic closed path in $E$.  Given  edges $f_1,..., f_\ell, g_1,..., g_m$ in $E$, if $0\not=f_1\cdots f_\ell c^\infty=g_1\cdots g_m c^\infty$ in $V_{[c^\infty]}$, then $f_1\cdots f_\ell-g_1\cdots g_m\in L_K(E)(c-1)$.
\end{lemma}
\begin{proof}
(1)  If $r(c-1)=0$, then $r=rc$ and hence $r=rc^m$ for each $m\geq 0$. Let $r=\sum_{i=1}^t k_i\alpha_i\beta_i^*$, with $\alpha_i,\beta_i$ real paths and $k_i\in K$. Denoting by $N$ the maximum length of the $\beta_i$'s, we have that $r=rc^N$ can be written as a $K$-linear combination $\sum_{i=1}^t k_i\gamma_i$ of real paths $\gamma_i$'s. Then, by a degree argument, from $r=rc$ we get $r=0$.\\

(2) and (3)    
Write $c = e_1 e_2 \cdots e_n$.  Assume
\[0\not=f_1\cdots f_\ell c^\infty=g_1\cdots g_mc^\infty\]
for some edges $f_1,...,f_\ell,g_1,...,g_m$.
If $\ell=m$ then $f_i=g_i$ for each $1\leq i\leq \ell=m$. If $m>\ell$, then there exists $j\in\mathbb N$ and $1\leq k\leq n$ such that
\[f_1\cdots f_\ell c^\infty=f_1\cdots f_\ell c^je_1\cdots e_kc^\infty=g_1\cdots g_mc^\infty \]
with $m=\ell+j\times n+k$ and $1\leq k\leq n$, $j\geq 0$. Then by the first equality we get $c^\infty=c^je_1\cdots e_kc^\infty$ and so $c^\infty=e_1\cdots e_kc^\infty$; hence $e_1\cdots e_k=c$ since $c$ is basic. Therefore $k=n$ and $g_1\cdots g_m=f_1\cdots f_\ell c^{j+1}$.   This contradicts the hypotheses in (2), so we have $m=\ell$ and $f_i = g_i$ for all $1\leq i \leq m$ in that case.   
%By hypothesis in (2), this then yields that $g_1\cdots g_m=f_1\cdots f_\ell$.  
Further, this yields 
\[g_1\cdots g_m-f_1\cdots f_\ell=f_1\cdots f_\ell (c^{j+1}-1)=f_1\cdots f_\ell (c^j+\cdots+c+1)(c-1)\in L_K(E)(c-1),\]
which gives (3).  %Thus we get $\overline{f_1\cdots f_\ell}=\overline{g_1\cdots g_m}$.
\end{proof}

\begin{lemma}\label{lemma:riassunto2}
Let $E$ be a finite graph, and $c=e_1\cdots e_n$ a basic closed path in $E$.
%Then $L_K(E)(c-1)$ is the kernel of $\hat\rho_c^\infty:L_K(E)\to V_{[c^\infty]}$, $r\mapsto rc^\infty$, and hence, 
Denoting  by $\rho_{c^\infty}:L_K(E)\to V_{[c^\infty]}$ the map $r\mapsto rc^\infty$ and by $\rho_{c-1}:L_K(E)\to L_K(E)$ the right multiplication by $c-1$,  we have the following short exact sequence of left $L_K(E)$-modules:
\[\xymatrix{0\ar[r]&L_K(E)\ar[r]^{\rho_{c-1}}&L_K(E)\ar[r]^{\rho_{c^\infty}}&V_{[c^\infty]}\ar[r]&0.}
\]
\end{lemma}
\begin{proof}
The map $\rho_{c-1}$ is a monomorphism by Lemma~\ref{lemma:riassunto1}(1), and $\rho_{c^\infty}$ is an epimorphism by construction.
Clearly $\Im \rho_{c-1}=L_K(E)(c-1)\subseteq \Ker \rho_{c^\infty}$.
Assume now 
%$c$ is basic and $r\in \Ker \hat\rho_c^\infty$. Let 
$r=\sum_{i=1}^t k_i\alpha_i\beta_i^*$ belongs to $\Ker \rho_{c^\infty}$, with $\alpha_i,\beta_i$ real paths and $k_i\in K$. Our aim is to prove that $\overline r=r+L_K(E)(c-1)=0$ and hence $r\in L_K(E)(c-1)$.
If $\alpha_i\beta_i^*c^m=0$ for a suitable $m\geq 1$, then $\alpha_i\beta_i^*=-\alpha_i\beta_i^*(c^m-1)=-\alpha_i\beta_i^*(1+\cdots+c^{m-1})(c-1)$ and hence $\overline{\alpha_i\beta_i^*}= 0$. Therefore we can assume $\alpha_i\beta_i^*c^m\not=0$ for all $m\geq 0$ and $1\leq i\leq t$. It follows that $\beta_i^*=e_{j_i}^*\cdots e_2^* e_1^*(c^{m_j})^*$ for suitable $1\leq j_i\leq t$ and $m_j\geq 0$. 
Since 
\[e_{j_i+1}\cdots e_n -e_{j_i}^*\cdots e_2^* e_1^*(c^{m_j})^*=e_{j_i}^*\cdots e_2^* e_1^*(c^{m_j})^*(c^{m_j+1}-1)=\]
\[=
e_{j_i}^*\cdots e_2^* e_1^*(c^{m_j})^*(c^{m_j}+c^{m_j-1}+\cdots+1)(c-1)\in L_K(E)(c-1),\]
we have
\[\overline r=\sum_{i=1}^t k_i\alpha_i \overline{e_{j_i}^*\cdots e_2^* e_1^*(c^{m_j})^*}=
\sum_{i=1}^t k_i\overline{\alpha_i e_{j_i+1}\cdots e_n}=\]
\[=\sum_{i=1}^s h_i\overline{f_1\cdots f_{j_i}}\]
where the $h_i$'s belongs to $K$ and the $\overline{f_1\cdots f_{j_i}}$ are distinct elements modulo $L_K(E)(c-1)$. Therefore by Lemma~\ref{lemma:riassunto1}(3)  the ${f_1\cdots f_{j_i}}c^\infty$ expressions  are distinct infinite paths which are tail equivalent to $c^\infty$, and hence linearly independent.
Since $0=rc^\infty=\sum_{i=1}^s h_i{f_1\cdots f_{j_i}}c^\infty$, we get $h_i=0$ for $1\leq i\leq s$ and hence
\[\overline r=\sum_{i=1}^s h_i\overline{f_1\cdots f_{j_i}}=0,\]
as desired.  
%%Now, assume
%%\[0\not=f_1\cdots f_\ell c^\infty=g_1\cdots g_mc^\infty\]
%%for some edges $f_1,...,f_\ell,g_1,...,g_m$.
%%If $\ell=m$ then $f_i=g_i$ for each $1\leq i\leq \ell=m$. If $m>\ell$, then there exists $j\in\mathbb N$ and $1\leq k\leq n$ such that
%%\[f_1\cdots f_\ell c^\infty=f_1\cdots f_\ell c^je_1\cdots e_kc^\infty=g_1\cdots g_mc^\infty \]
%%with $\ell=m+j\times n+k$ and $1\leq k\leq n$, $j\geq 0$. Then by the first equality we get $e_1\cdots e_kc^\infty=c^\infty$ and hence $e_1\cdots e_k=c$ since $c$ is basic. Therefore $k=n$ and $g_1\cdots g_m=f_1\cdots f_\ell c^{j+1}$; hence
%%\[g_1\cdots g_m-f_1\cdots f_\ell=f_1\cdots f_\ell (c^{j+1}-1)=f_1\cdots f_\ell (c^j+\cdots+c+1)(c-1)\in L_K(E)(c-1).\]
%%Thus we get $\overline{f_1\cdots f_\ell}=\overline{g_1\cdots g_m}$.\\
%
%
%
%
%
%$0\not=\alpha_{i_1} e_{j_{i_1}+1}\cdots e_nc^\infty=\alpha_{i_2} e_{j_{i_2}+1}\cdots e_nc^\infty$ implies
%\[\alpha_{i_1} e_{j_{i_1}+1}\cdots e_n=\alpha_{i_2} e_{j_{i_2}+1}\cdots e_nc^{m_{1,2}}\text{ or }\]
%\[\alpha_{i_1} e_{j_{i_1}+1}\cdots e_nc^{m_{2,1}}=\alpha_{i_2} e_{j_{i_2}+1}\cdots e_n.\]
%
%
%
%
%
%
%We have
%\[0=rc^\infty=\left(\sum_{i=1}^n k_i\alpha_i e_{j_i}^*\cdots e_2^* e_1^*c^{m_j}^*\right)c^\infty=\sum_{i=1}^n k_i\alpha_i 
%e_{j_i+1}\cdots e_n c^\infty
%
%=======
%Following the proofs of \cite[Lemma~2.7 and Theorem 2.8]{AMT1}, one realise that the hypotheses can be relaxed to the case $c$ is a basic closed path.
\end{proof}

The short exact sequence established  in Lemma~\ref{lemma:riassunto2} provides a projective resolution for the Chen simple module $V_{[c^\infty]}$. In particular, we get 

\begin{corollary}\label{Vcinftyisotoquotient} 
Let $c$ be a basic closed path in $E$.  Then $ L_K(E)/L_K(E)(c-1)$ is isomorphic to the Chen simple $L_K(E)$-module  $V_{[c^\infty]}$. 
\end{corollary}

%introduction to $V^f_{c^\infty}$+$V^f_{c^\infty}\not\cong V^g_{d^\infty}$ if $(f,c)\not= (g,d)$???.
%We will concentrate on $f(x)=x-1$, leaving to a short appendix the general case.
%
%An infinite path $p=e_1e_2\cdots$ is tail equivalent to $c^\infty$ if and only if there exists $i_p\in\mathbb N$ such that $e_{i_p}e_{i_p+1}\cdots=c^\infty$. Let us denote by $V_{[c^\infty]}$ the set of all $K$-linear combinations of infinite paths tail equivalent to $c^\infty$: $V_{[c^\infty]}$ is called the Chen simple module associated to $c$. 

%Recalling the proof of \cite[Lemma~2.7]{AMT1}, it is easy to check that  right multiplication by $c-1$ gives a monomorphism $\hat\rho_{c-1}:L_K(E)\to L_K(E)$ whenever $c$ is a closed path.  Moreover, by \cite[Proposition 2.6]{AMT1}, we have $\Ker\hat\rho_{c^\infty}=\Im \hat\rho_{c-1}=L_K(E)(c-1)$.  
%
%Prove for closed path $c$ which are not powers of closed path that:
%1. $\rho_{c^\infty}:L_K(E)\to V_{c^\infty}$ has kernel $L_K(E)(c-1)$ (look at Lemma 2.5, observing that the proof continues to work).
%2. $\rho_{c-1}:L_K(E)\to L(K(E)$ is a monomorphism always.

\section{A division algorithm in $L_K(E)$}\label{divisionalgorithmsection}

Let $c$ be a basic closed path in $E$. In this section we show how any element of $L_K(E)$ may be ``divided by"  $c-1$, in an analogous  manner to the standard division algorithm in $\mathbb{Z}$.   

\begin{definition}\label{AcDefinition}
{\rm Let $E$ be any finite graph, and $c$ any basic closed path in $E$ of length $> 0$ with $v=s(c)$. 
%%%%%%%%the paragraph holds for any simple closed path, not only for loops!. 
We denote by $A_c$  the set of all non-vertex real paths $\alpha$ in $E$ 
%$L_K(E)$ 
which are not divisible by $c$ either on the left or on the right, but are non trivially composable with $c$ on the right.  Formally:
\[A_c=\{\alpha\in\text{Path}(E):  \  |\alpha| \geq 1; \ \alpha \neq \beta c;   \text{ and } \alpha \not=c\gamma \ \text{for any real paths $\beta, \gamma$}, \text{and }r(\alpha) = v\} .\]
For each $i\in\mathbb N_{\geq 1} $
%= \{0,1,2,\dots\}$, 
we denote by $c^iA_c$ the subset $\{c^i\alpha:\alpha\in A_c\}$ of elements of $L_K(E)$. We understand
%
%of paths in $E$ of the form $\{c^i\alpha:\alpha\in A_c\}$. 
 %$c^0A_c$ denotes $A_c$ itself and 
 $c^iA_c=\emptyset$ whenever $A_c=\emptyset$. 
%\bigskip
%
%{\bf A/F:  What is $c^0 A_c$ to denote?}
%
%\bigskip
We denote by $G$ the $K$-vector subspace of $L_K(E)$ generated by $1_{L_K(E)}$, the elements in $A_c$ and the elements in 
$c^iA_c$, $i\in\mathbb N_{\geq 1}$.  That is, 
\[G:= K[1_{L_K(E)}, A_c, \bigcup_{i\in\mathbb N_{\geq 1}}c^iA_c].\]
}
\end{definition}
%
%
%{\bf XXX where do we use the hypothesis that $c$ is basic?}
%\textcolor{red}{In the definition of the map $\sigma$ below: if $d=c^2$ then $c$ belongs to $A_d$ and since $p:=cd^\infty=d^\infty$, $\sigma(p)$ is not well defined ...}
%
%
%
\begin{example}\label{AcExample}
\begin{enumerate}
\item Let $E$ be the graph
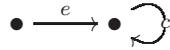
\[\xymatrix{\bullet\ar[r]^e&\bullet\ar@(ur,dr)[]|{c}}\]
Then $A_{c}=\{e\}$ and $c^nA_{c}=\{0\}$ for each $n\geq 1$. Then $G$ is the two dimensional vector space generated by $1$ and $e$.
\item Let $E=R_1$, the rose with one petal:
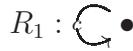
\[R_1:\xymatrix{\bullet\ar@(ul,dl)[]|{c}}\]
Then $A_{c}=\emptyset$ (and so also $c^nA_{c}=\emptyset$ for each $n\geq 1$). Then $G$ is the one dimensional vector space generated by $1$.
\item Let $E=R_2$, the rose with two petals:
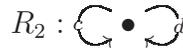
\[R_2:\xymatrix{\bullet\ar@(ul,dl)[]|{c}\ar@(ur,dr)[]|{d}}\]
Then $A_{c}=\{d^ic^jd^k:i,k\in\mathbb N_{\geq 1}, j\in\mathbb N\}$ and $c^nA_{c}=\{c^nd^ic^jd^k:i,k\in\mathbb N_{\geq 1}, j\in\mathbb N\}$ for each $n\geq 1$. Then $G$ is a countable dimensional vector space.
%
%
%
%Let $E_i=R_i$, $i=1,2$, the rose with $i$ petals. 
%\[R_1:\xymatrix{\bullet\ar@(ul,dl)[]|{c_1}}\qquad R_2:\xymatrix{\bullet\ar@(ul,dl)[]|{c_2}\ar@(ur,dr)[]|{d}}\]
%Then $A_{c_1}=\emptyset$ and $A_{c_2}=\{d^ic_2^jd^k:i,k\in\mathbb N_{\geq 1}, j\in\mathbb N\}$.
%Moreover $c_1^nA_{c_1}=\emptyset$ and $c_2^nA_{c_2}=\{c_2^nd^ic_2^jd^k:i,k\in\mathbb N_{\geq 1}, j\in\mathbb N\}$ for each $n\in \mathbb N$.
\end{enumerate}
\end{example}

\begin{remark}\label{rem:Gincasecissource}
Clearly the non-zero elements in
$\{1_{L_K(E)}\}\cup A_c\cup \bigcup_{i\in\mathbb N_{\geq 1}}c^iA_c$ form
a $K$-basis for $G$.
%
%$G$ is a $K$-vector space of less or equal than countable dimension with basis the non-zero elements in
%$\{1_{L_K(E)}\}\cup \bigcup_{i\in\mathbb N}c^iA_c$.
%
%
%\bigskip
%
%
%
%
%{\bf A/F ... we should decide whether we want to view $A_c$ as paths in $E$, or as elements in $L_K(E)$.  Of course these have both roles, but we should be consistent with the language.   }
%
%\bigskip
%
Therefore a generic element $g$ in $G$ is of the form
\[g=k1_{L_K(E)}+t_{1}+ct_{2}+c^2t_{3}+\cdots+c^{s-1}t_{s}\]
where $k\in K$ and $t_i$ are $K$-linear combinations in $L_K(E)$ of elements in $A_c$. It is convenient to refer to $k1_{L_K(E)}$ as the {\it scalar part}   of $g$: the latter commutes with any element in $L_K(E)$.\\
If  $c$ is a source loop, then $A_c = \emptyset$ and $c^iA_c=\emptyset$ for all $i\geq 1$: therefore $G$ is the one-dimensional $K$-vector subspace of $L_K(E)$ generated by $1_{L_K(E)}$.\\
If $c=e_1\cdots e_n$ is a source cycle, then $A_c=\{e_n, e_{n-1}e_n,..., e_2e_3\cdots e_n\}$ and $c^iA_c=\{0\}$ for each $i\ge 1$.
%
%{\bf A/F:  Depending on our point of view, maybe $c^i A_c = \emptyset$, rather than $0$?}   
%{\bf G: For us $c^i A_c =0$}
%
Therefore $G$ is the $K$-vector subspace of $L_K(E)$ of dimension $n$ generated by $1_{L_K(E)}$, and the paths $e_n, e_{n-1}e_n,..., e_2e_3\cdots e_n$.\\
In general $G$ is a finite dimensional space if and only if $A_c$ is finite and $cA_c$ is zero or empty. This happens if and only if there are no cycles different from $c$ connected to $s(c)$, i.e., when $c$ is a maximal cycle.  
\end{remark}

\begin{definition}\label{sigmadef}
{\rm  Let $c$ be a basic closed path in $E$.   As above, we denote by $\rho_{c^\infty}: L_K(E)\to V_{[c^\infty]}$ the right multiplication by $c^\infty$  homomorphism.  
%introduced in Lemma~\ref{lemma:riassunto2}.   
By Lemma \ref{lemma:riassunto1}(2), 
each infinite path $p$ tail equivalent to $c^\infty$ uniquely determines an element of $\{1_{L_K(E)}\}\cup A_c\cup
\left(\bigcup_{i\in\mathbb N_{\geq 1}}c^iA_c\right)$, which we denote by $\sigma(p)$.  Specifically, $\sigma(p)$ has the property that 
%, having the property that 
%\[\sigma(p)\in \{1_{L_K(E)}\}\cup A_c\cup
%\left(\bigcup_{i\in\mathbb N_{\geq 1}}c^iA_c\right)\]
 $$p=\sigma(p)c^\infty=\rho_{c^\infty}(\sigma(p)).$$ 
 }
\end{definition}

%{\bf XXX We previously used the letter $\sigma$ to denote a cycle.  Maybe we should change that.}
%\textcolor{red}{DONE}

%\bigskip
%
%{\bf  I am not sure what side to write these maps on.  They are $K$-vector space maps.  But eventually $\hat\rho_{c^\infty}$ will be extended to an $R$-homomorphism.   We can decide that later.}
%
%\bigskip

Extending $\sigma$ by linearity, the maps
$$\sigma: V_{[c^\infty]} \to G \ \  \ \ \mbox{and} \ \ \ \ \  \rho_{c^\infty|G} : G \to V_{[c^\infty]}$$
are then easily seen to be inverse isomorphisms of $K$-vector spaces. 
%Accordingly,  when we think of $V_{[c^\infty]}$ as a $K$-vector space,  we will typically represent it by $G$.

\begin{lemma}\label{zerointersection}    Let $c$ be a basic closed path in $E$.  Then   $L_K(E)(c-1)\cap G = \{0\}.$ 
\end{lemma}
\begin{proof}
If $\ell = \ell_0 (c-1) \in L_K(E)(c-1)\cap G$, then 
$\ell=\sigma(\rho_{c^\infty}(\ell))$ (by the previous observation, as $\ell \in G$), which in turn equals $\sigma(\rho_{c^\infty}(\ell_0(c-1)))=\sigma(\ell_0(c-1)c^\infty)=\sigma(\ell_0(c^\infty - c^\infty)) = \sigma(0)=0.$
\end{proof}

%{\bf XXX do we need that $c$ is basic to get the Division Algorithm?}
%\textcolor{red}{Yes! We need the map $\sigma$!!!}

\begin{theorem}[Division Algorithm by $c-1$]\label{thm:division}  Let $E$ be any finite graph and $K$ any field.  Let $c$ be a \textcolor{black}{basic closed path} in $E$.   Then 
for any $\beta\in L_K(E)$ there exist unique $q_\beta\in  L_K(E)$ and $r_\beta\in G$ such that
\[\beta=q_\beta(c-1)+r_\beta.\]
\end{theorem}
\begin{proof}
%In \cite[Theorem~2.8]{AMT1} we have proved the existence of the following short exact sequence:
%\[\xymatrix{0\ar[r]& L_K(E)\ar[r]^{\hat\rho_{c-1}}&L_K(E)\ar[r]^{\hat\rho_{c^\infty}}&V_{[c^\infty]}\ar[r]&0}.\]
Consider the element $r_\beta:=\sigma\left(\rho_{c^\infty}(\beta)\right)$; clearly $r_\beta$ belongs to $G\subseteq L_K(E)$. The difference $\beta-r_\beta$ belongs to $\Ker\rho_{c^\infty}$, as
 \[\rho_{c^\infty}(\beta-r_\beta)=\beta c^\infty-r_\beta c^\infty=
\beta c^\infty-\sigma\left(\rho_{c^\infty}(\beta)\right) c^\infty 
=%\beta c^\infty-\hat\rho_{c^\infty}(\beta)=
\beta c^\infty-\beta c^\infty=0.\]
Since $\Ker\rho_{c^\infty}=L_K(E)(c-1)$ by Lemma~\ref{lemma:riassunto2}, we have  
%which as mentioned above equals $L_K(E)(c-1)$, so that   
$\beta-r_\beta = q_\beta(c-1)$ for a suitable $q_\beta\in L_K(E)$.   
Let us prove that $q_\beta\in L_K(E)$ and $r_\beta\in G$ are uniquely determined. Assume
\[\beta=q_1(c-1)+r_1=q_2(c-1)+r_2;\]
then we have $r_1-r_2=(q_2-q_1)(c-1)\in L_K(E)(c-1)\cap G$, which is $0$ by Lemma~\ref{zerointersection}.
Therefore $r_1=r_2$ and $\rho_{c-1}(q_2-q_1)=(q_2-q_1)(c-1)=r_1-r_2=0$; since $\rho_{c-1}$ is a monomorphism by Lemma~\ref{lemma:riassunto2}, we have $q_1=q_2$.
\end{proof}

Here are two specific applications of the Division Algorithm by $c-1$, both of which will be quite useful in the sequel.  
\begin{example}\label{rem:c^n}
{\rm  Since
\[c^n=(1+(c-1))^n=\sum_{j=0}^n\binom nj(c-1)^j,\]
by Theorem \ref{thm:division} we deduce
$q_{c^n} = \sum_{j=1}^n\binom nj(c-1)^j$, and $r_{c^n} = 1$.
%
%
%
% In the particular case where $\beta = c^n$, then dividing $\beta$ by $(c-1)$ as in  Theorem \ref{thm:division} gives 
%\[c^n=(1+(c-1))^n=\sum_{j=0}^n\binom nj(c-1)^j.\]
%In particular, $q_{c^n} = \sum_{j=1}^n\binom nj(c-1)^j$, and $r_{c^n} = 1$.\\
}
\end{example}
\begin{example}\label{DivAlgforG}
{\rm  We will have need to multiply various elements of $L_K(E)$ on the  left by $c-1$.   
 Let $g=k1_{L_K(E)}+t_{1}+ct_{2}+c^2t_{3}+ \cdots +c^{s-1}t_{s}$ be an arbitrary  element of $G$.  Then multiplying and collecting appropriate terms yields
\[(c-1)g=  k(c-1)   -t_1+c(t_1-t_{2})+c^2(t_2-t_{3})+\cdots+c^{s-1}(t_{s-1}-t_{s})+c^st_s.\]
So by the uniqueness part of the Division Algorithm, we get
$$q_{(c-1)g}=k1_{L_K(E)},  \mbox{ and }   r_{(c-1)g}=-t_1+c(t_1-t_{2})+c^2(t_2-t_{3})+\cdots+c^{s-1}(t_{s-1}-t_{s})+c^st_s.$$
Note in particular that the scalar part of $r_{(c-1)g}$ is $0$.   
}
\end{example}

\begin{remark}
If $E=R_1$ is the rose with one petal $c$, then $L_K(E)\cong K[x,x^{-1}]$ via $c\mapsto x$. In such a case the above Division Algorithm with respect to $c-1$ corresponds to the usual division by $x-1$.
\end{remark}

%{\bf XXX do we want to be consistent on our use of $1$ vs. $1_{L_K(E)}$?  }
%\textcolor{red}{The idea is to use $1_{L_K(E)}$ only when there are more rings involved, like in the Morita equivalences}

%%%%%%%%%%%%%%%%%%%%%%%%%%%%
%
%
%%%%% \section{The modules $\Muni {i}{L_K(E)}{c-1}$ and their $G$-representation}
%
%
%%%%%%%%%%%%%%%%%%%%%%%%%%%%%

\section{The Pr\"ufer modules $U_{L_K(E), c-1}$}
%\section{The modules $\Muni {n}{L_K(E)}{c-1}$ and their $G$-representation}
\label{Umodulessection} 
%
%
%Recall that for $r \in L_K(E)$ we let $\langle r \rangle$ denote the left ideal $L_K(E)r$. 

Let $c$ be a basic closed path in $E$. By Lemmas~\ref{lemma:riassunto1}(1)  and  \ref{lemma:riassunto2}, the element $c-1$ is neither a right zero divisor, nor left invertible.   Therefore we can apply the construction of the Pr\"ufer module given in Section~\ref{sec:generalring} for $R = L_K(E)$ and $a = c-1$. For efficiency, in the sequel we use the following notation.  
$$\Muni {n}{E}{c-1}:=\Muni {n}{L_K(E)}{c-1}; \ \ \psi_{E,i,j}:=\psi_{L_K(E),i,j}; \ \ U_{E,c-1}:=U_{L_K(E),c-1}; \ \ \mbox{and} \ \psi_{E,i}:=\psi_{L_K(E),i}.$$

Most importantly for us, by Corollary \ref{Vcinftyisotoquotient}  $\Muni 1E{c-1} = L_K(E) / L_K(E)(c-1)$ is simple, indeed, is isomorphic to the Chen simple module $V_{[c^\infty]}$.

%{\bf XXX In order to see that the  $\Muni 1E{c-1}$ are uniserial we need to show that the appropriate equation from Proposition \ref{prop:uniserialchar} has no solution, and we haven't done that yet.   So I have moved things around a little bit.    ...}  
%\textcolor{red}{We postponed also the description of the generators $\alpha_i$}

%It could be an idea to write in this section how to represent the elements in  $M_n=R/{<(c-1)^n>}$, using the division algorithm. Then discuss how to solve equations in $M_n$ using this representation.

%\medskip

For the sequel, it is useful to have a canonical representation of the elements of the uniserial modules $\Muni n E{c-1}$, $n\geq 1$.

\begin{proposition}\label{prop:Grep}
Let $c$ be a basic closed path in $E$, $n\in \mathbb{N}$ and $x\in \Muni {n}{E}{c-1}$.  Then $x$ can be written in a unique way as $$x = g_1+g_2(c-1)+\cdots+g_n(c-1)^{n-1}+L_K(E)(c-1)^n$$ with the $g_i$'s belonging to $G$. %and uniquely determined. 
We call the displayed expression  the $G${\rm -representation} of $x$.
\end{proposition}
%\medskip

%{\bf $G$-representation for $M_2$ (and $M_n$)}:

\begin{proof}   
Assume $x\in \Muni {n}{E}{c-1}$, and write $x=y+L_K(E)(c-1)^n$ with $y\in L_K(E)$. Then invoking Theorem \ref{thm:division} $n$ times we have $y=q_{1}(c-1)+g_{1}$, $q_{1}=q_{2}(c-1)+g_{2}$, ..., and $q_{n-1}=q_{n}(c-1)+g_{n}$.  Therefore
 \[x=y+L_K(E)(c-1)^n=g_{1}+g_{2}(c-1)+\cdots+g_{n}(c-1)^{n-1}+L_K(E)(c-1)^n,\]
where the elements $g_{i}$, $i=1,...,n$, belong to $G$.
 Assume now $x=g'_{1}+g'_{2}(c-1)+\cdots+g'_{n}(c-1)^{n-1}+L_K(E)(c-1)^n$, where $g'_{i}$, $i=1,...,n$, belong to $G$.  Then  $(g_1 - g_1') + (g_2-g_2')(c-1)+\cdots+ (g_n-g_n')(c-1)^{n-1}$ belongs to $L_K(E)(c-1)^n$. Therefore
$ g_1 - g_1'$ belongs to $L_K(E)(c-1)\cap G=0$ (by Lemma \ref{zerointersection}), and hence $g_1=g_1'$. Since multiplication by $c-1$ on the right is a monomorphism, we get that $(g_2-g_2') + (g_3 - g_3')(c-1) +\cdots+ (g_n-g_n')(c-1)^{n-2}$ belongs to $L_K(E)(c-1)^{n-1}$; therefore also 
 $g_2-g_2'$ belongs to $L_K(E)(c-1)\cap G=0$, and hence $g_2=g_2'$. Repeating the same argument we get $g_i=g_i'$ for $i=1,...,n$.
\end{proof}

\begin{example}
If $E=R_1$ and hence $L_K(E)\cong K[x,x^{-1}]$, then $\Muni {n}{R_1}{c-1}\cong K[x,x^{-1}]/\langle(x-1)^n\rangle$. For instance, the $G$ representation of $x^{-4}+2+x+K[x,x^{-1}](x-1)^3$ can easily be shown to be
\[4-3(x-1)+10(x-1)^2+K[x,x^{-1}](x-1)^3.\]
\end{example}

     %So by   Proposition \ref{prop:uniserialchar},  each $\Muni nE{c-1}$ has a unique composition series of length $n$ with composition factors isomorphic to $\Muni 1E{c-1}\cong V_{[c^\infty]}$.

    We are now in position to show that the  modules $\Muni {i}E{c-1}$, $i\geq 1$, satisfy the hypotheses of Propositions \ref{prop:uniserialchar} and \ref{genUisuniserialandartinian}.

\begin{proposition}\label{nosolutionsto(c-1)YinMn}
For any basic closed path $c$ in $E$, the equation 
$$(c-1)\mathbb X = 1 + L_K(E)(c-1)^n$$
has no solution in $\Muni {n}{E}{c-1}$.  
\end{proposition}
\begin{proof}
  By Proposition \ref{prop:Grep}, we have to verify that the following equation in the $n$ variables $X_1$,..., $X_n$ does not admit solutions in $G^n$ (the direct product of $n$ copies of $G$):
 \[(c-1)(X_1 + X_2(c-1) + \cdots + X_n(c-1)^{n-1}+L_K(E)(c-1)^n)=1+L_K(E)(c-1)^n.\]
Assume on the contrary that $X_i=g_i$ (for $i=1,..., n$) is a solution. Let $k_i1_{L_K(E)}$ be the scalar part of $g_i$. Since $(c-1)g_i=k_i(c-1) + g'_i$ for a suitable $g'_i\in G$ whose scalar part is zero (see Lemma \ref{DivAlgforG}), we would have 
  \begin{align*}
  1+L_K(E)(c-1)^n
&=(c-1)(g_1+ g_2(c-1) + \cdots+g_n(c-1)^{n-1}+L_K(E)(c-1)^n) \\
&=  g'_1+(g'_2 +k_1)(c-1)+ \cdots+ (g'_n+k_{n-1})(c-1)^{n-1}+L_K(E)(c-1)^n.
\end{align*}
By Proposition~\ref{prop:Grep} the $G$-representation of each element  of $\Muni {n}{E}{c-1}$ is unique. Therefore we would have that $g'_1=1_{L_K(E)}$ has non zero scalar part, a contradiction.
\end{proof}

%\begin{corollary}\label{cor:Uniserial}
%If $c$ is a simple closed path, then $M_n(c)$ is uniserial of length $n$ for each $n\in\mathbb N$ whose composition factors are isomorphic to the Chen simple module $V_{[c^\infty]}$. 
%For each $m\leq n$ we have the following canonical inclusion:
%\[\psi_{m,n}:M_m(c)\to M_n(c),\quad 1+\langle (c-1)^m\rangle\mapsto (c-1)^{n-m}+\langle (c-1)^n\rangle.\]
%\end{corollary}
%\begin{proof}
%If $c$ is a simple closed path, then $L_K(E)/<(c-1)>$ is a simple Chen module. We conclude by Lemma~\ref{nosolutionsto(c-1)YinMn} and Proposition~\ref{prop:uniserialchar}.
%The last inclusion follows by general consideration: see Section~\ref{sec:generalring}.
%\end{proof}
So Corollary \ref{Vcinftyisotoquotient}  and     Proposition \ref{nosolutionsto(c-1)YinMn}   combine with  Propositions \ref{prop:uniserialchar} and \ref{genUisuniserialandartinian} to immediately yield the  following key result.
  
\begin{theorem}\label{Pruferisuniserialandartinian}
Let  $c$ be a basic closed path in $E$.  

1)  For each $n\in \N$,  the $L_K(E)$-module    $\Muni {n}{E}{c-1}$ has a unique composition series, with all composition factors isomorphic to $V_{[c^\infty]}$. 

2)     The Pr\"ufer  $L_K(E)$-module $\Uni E{c-1}$ is uniserial  and artinian (and not noetherian).
\end{theorem}

\textcolor{black}{The left $L_K(E)$-module $\Uni E{c-1}$ is generated by the elements 
$\alpha_i:=\psi_{E,i}(1+L_K(E)(c-1)^i)$, which satisfy\[(c-1)\alpha_i=\begin{cases}
     0 & \text{if }i=1, \\
    \alpha_{i-1}  & \text{if }i>1.
\end{cases}
\]
}

\begin{remark} \label{remark:levels}
By Proposition~\ref{nosolutionsto(c-1)YinMn},  the equation 
\[(c-1)\mathbb X=1+L_K(E)(c-1)^{n}\] 
has no solution in $\Muni {n}{E}{c-1}$. But identifying $\Muni {n}{E}{c-1}$ with its copy $\psi_{E,n,n+1}(\Muni {n}{E}{c-1})$ in $\Muni{n+1}{E}{c-1}$, the same equation has the form
\[(c-1)\mathbb X=(c-1)+L_K(E)(c-1)^{n+1},\]
which  clearly  admits the solution $\mathbb X=1+L_K(E)(c-1)^{n+1}$. This observation will be crucial to study the injectivity of the Pr\"ufer modules discussed in the following section.   
\end{remark}

\textcolor{black}{If $c'$ is a cyclic shift of the basic closed path $c$, then it is clear that $V_{[c^\infty]}=V_{[{c'}^\infty]}$.}
We conclude the section with a perhaps-not-surprising result which shows that the  cyclic shift of a basic closed path does not affect the isomorphism class of the corresponding Pr\"ufer module.

%Assume $c=e_1e_2\cdots e_n$ with $n\geq 2$,. We denote by $c_i$ the cycle $e_i\cdots e_ne_1\cdots e_{i-1}$; in particular $c=c_1$.
%
\begin{proposition}
Let $c=e_1e_2\cdots e_n$ with $n\geq 2$ be a basic closed path. Denote by $c_i$ the basic closed path $e_i\cdots e_ne_1\cdots e_{i-1}$. Then the modules $\Muni {m}{E}{c-1}$ and $\Muni {m}{E}{c_\ell-1}$ are isomorphic for all $1\leq \ell\leq n$ and $m\in\mathbb N_{\geq 1}$. In addition, the corresponding Pr\"ufer modules $\Uni E{c-1}$ and $\Uni E{c_\ell-1}$ are isomorphic  for all $1\leq \ell\leq n$.
\end{proposition}
\begin{proof}
It is easy to verify that $(c_1-1)e_1\cdots e_{\ell-1}=e_1\cdots e_{\ell-1}(c_\ell-1)$ and
$(c_\ell-1)e_\ell\cdots e_n=e_\ell \cdots e_{n}(c_1-1)$.
So the maps
$\varphi_{1,\ell}: \Muni {m}{E}{c_1-1}\to \Muni {m}{E}{c_\ell-1}$ and $\varphi_{\ell,1}:\Muni {m}{E}{c_\ell-1}\to \Muni {m}{E}{c_1-1}$ given by
\[\varphi_{1,\ell}:1+L_K(E)(c_1-1)^m \mapsto e_1\cdots e_{\ell-1}+L_K(E)(c_\ell-1)^m,\quad \text{and}\]
\[
\varphi_{\ell,1}:1+L_K(E)(c_\ell-1)^m\mapsto (-1)^{m-1}e_\ell\cdots e_n\sum_{i=1}^m\binom mi (-1)^{m-i}c_1^{i-1}
+L_K(E)(c_1-1)^m\]
are well defined. Let us prove that they are inverse isomorphisms.
Denote by $\overline r$ both the cosets $r+L_K(E)(c_1-1)^m$ and $r+L_K(E)(c_\ell-1)^m$. Then
\begin{align*}
\varphi_{\ell,1}\circ \varphi_{1,\ell}(\overline{1}) & =\varphi_{\ell,1}(\overline{e_1\cdots e_{\ell-1}}) \\
& =\overline{(-1)^{m-1}c_1\sum_{i=1}^m\binom mi (-1)^{m-i}c_1^{i-1}} \\
& =\overline{(-1)^{m-1}\sum_{i=1}^m\binom mi (-1)^{m-i}c_1^{i}} \\
&= \overline{(-1)^{m-1}\big((c_1-1)^m-(-1)^m\big)} \\
& = \overline{(-1)^{m-1}(-(-1)^m)}\\ 
&  = \overline{1}.
\end{align*}
Analogously
\begin{align*} 
\varphi_{1,\ell}\circ\varphi_{\ell,1}(\overline{1}) &= \varphi_{1,\ell}\big(\overline{(-1)^{m-1}e_\ell\cdots e_n\sum_{i=1}^m\binom mi (-1)^{m-i}c_1^{i-1}}\big) \\
& =\overline{(-1)^{m-1}e_\ell\cdots e_n\sum_{i=1}^m\binom mi (-1)^{m-i}c_1^{i-1}e_1\cdots e_{\ell-1}} \\
&= \overline{(-1)^{m-1}e_\ell\cdots e_ne_1\cdots e_{\ell-1}\sum_{i=1}^m\binom mi (-1)^{m-i}c_\ell^{i-1}} \\
&=\overline{ (-1)^{m-1}\sum_{i=1}^m\binom mi (-1)^{m-i}c_\ell^{i}} \\
& =\overline{(-1)^{m-1}\big((c_\ell-1)^m-(-1)^m\big) }\\
& =\overline{ (-1)^{m-1}(-(-1)^m)} \\
& =\overline{1}.   
\end{align*} 
Again using the initial observation, it is straightforward to check the commutativity of the appropriate diagrams, which gives the second statement.
\end{proof}

%\bigskip

%{\bf A/F:   I'm not sure exactly what we are trying to say in the previous remark.}

%\bigskip

%%%%%%%%%%%%%%%%%%%%%%%%
%
%    InjectivePrufersection
%
%%%%%%%%%%%%%%%%%%%%%%%%%%

\section{Conditions for injectivity of the Pr\"ufer modules $\Uni E{c-1}$}\label{InjectivePrufersection}

Let $E$ be any finite graph,  and let $c$ denote a basic closed path in $E$.

Of course the module $\Uni E{c-1}$ mimics in many ways the classical,  well-studied  Pr\"ufer groups from abelian group theory (see Example~\ref{ex:abeliangroups}).   It is well-known that the Pr\"ufer groups are divisible $\mathbb{Z}$-modules, hence injective.    With that observation as motivation, we study in the sequel the question of whether the Pr\"ufer modules $\Uni E{c-1}$ for various basic closed paths $c$ are injective $L_K(E)$-modules. The discussion will culminate in Theorem \ref{thm:injective}, characterizing the injectivity solely in terms of how the basic closed path $c$ sits in the graph $E$.

%{\bf XXX We should note somewhere that $\rho_{c-1}$ gives a hom from $\Uni E{c-1}$ to itself, because we use that in the next result. }
%\textcolor{red}{No; it is not clear if $\rho_{c-1}$ gives a hom from $\Uni E{c-1}$ to itself. We have changed the map in the proof.}

\begin{proposition}\label{Ucnotpureinjective} 
Let $E$ be a finite graph, let $c$ be a basic closed path in $E$ based at $s(c) = v$, and let $\Uni E{c-1}$ be the Pr\"ufer module associated to  $c$.
Suppose that there exists a cycle $d\not=c$ which connects to $v$.
%
%
%are infinitely many distinct paths $p$ in $A_c$ for which $r(p)=v$.
% $ \alpha c \neq 0$ for an infinite number of distinct  paths $\alpha  \in A_c$.    
Then  $\Uni E{c-1}$ is not injective.
\end{proposition}
\begin{proof}
%As $E$ is finite,  the condition yields that there necessarily exists a cycle  $d$ in $E$ and a path $\beta$ for which $d \beta$ belongs to $A_c$ for any $i\geq 1$. Therefore is a path in $E$, i.e., for which $d \beta c \neq 0$ in $L_K(E)$. 
%  such that $d'gc\neq 0$ (for some real path $g$).  
 % and 
 % \bigskip
% need to change this notation   Quit here Friday December 2
 %$s(d')\in U(V_{[c^{\infty}]})$.
 % \bigskip
 %
 The set 
 %$U(V_{[c^\infty]})$ 
 of those vertices of $E$ which are connected to $v$ contains the source of $d$. Therefore by \cite[Theorem~3.10]{AMT1},  $\Ext^1(V_{[{d}^{\infty}]}, V_{[c^{\infty}]})\neq 0$.  \textcolor{black}{Utilizing Remark~\ref{rem:quotientofU}, we get the exact sequence 
\[\xymatrix{0\ar[r] \ & \ V_{[c^{\infty}]}\cong L_K(E)\alpha_1               \ar@{^(->}[r]                  \ & \ \Uni E{c-1}       \ar@{->>}[r]        \ & \  \Uni E{c-1} /  L_K(E)\alpha_1 \cong     \Uni E{c-1} \ar[r] \ & \ 0 \ .}\]} 
\noindent
We have  $\Hom( V_{[{d}^{\infty}]}, \Uni E{c-1})=0$, because   the only simple submodule of $\Uni E{c-1}$ is isomorphic to $V_{[c^{\infty}]} \not\cong V_{[{d}^{\infty}]}$ (see Section~\ref{section:Chen}).  
%
%\bigskip
%
%{\bf A/F:   we should quote some result about why those are not isomorphic.}
%
%\bigskip
%
Consequently, the standard long exact sequence for $\Ext^1$ gives that $\Ext^1(V_{[{d}^{\infty}]}, \Uni E{c-1})\neq 0$, so that $\Uni E{c-1}$ is not injective, as claimed. 
\end{proof}

\begin{example}\label{VcinftyinR2notinjective}
{\rm 
\begin{enumerate}
\item Let $E = R_n$ be the graph with one vertex and $n$ loops.
%\begin{enumerate}
%\item If $n=1$, then the  Pr\"ufer module $\Uni E{c-1}=\varinjlim K[x,x^{-1}]/\langle(x-1)^n\rangle$ is the injective envelope of the simple $K[x,x^{-1}]$-module $K[x,x^{-1}]/\langle x-1\rangle\cong K$.
%\item 
If $n\geq 2$, then for any basic closed path $c$ the  Pr\"ufer module $\Uni E{c-1}$ is not injective. Indeed we can always find a loop different from $c$ which connects to $s(c)$.
\item If $c$ is a basic closed path which is not a cycle, then the  Pr\"ufer module $\Uni E{c-1}$ is not injective. Indeed 
there exists a cycle $d$ such that $c=\alpha d \beta$ with $\alpha$, $\beta\in \text{Path}(E)$, and at least one of $\alpha, \beta$ is not a vertex.  Clearly $d$ is connected to $s(c)$.
%
%{\bf XXX The previous statement needs more thought, it is likely true but the proof here assumes that $\beta$ is not a vertex.}
%\textcolor{red}{No; if $\beta$ is a vertex, then $r(d)=r(c)=s(c)$ and hence $d$ connects to $s(c)$.}
%
%
\end{enumerate}
}
\end{example}

%\begin{remark}
%If $c$ is a simple closed path which is not a cycle, then the  Pr\"ufer module $\Uni E{c-1}$ is not injective. Indeed if $c=e_1\cdots e_h\cdots e_j\cdots e_n$ with $s(e_h)=s(e_j)$, then $\alpha_m:=e_1\cdots (e_h\cdots e_{j-1})^m e_j\cdots e_n$ belongs to $A_c$  and $\alpha_mc\not=0$ for each $m\geq 2$.
%\end{remark}

By (2) of the previous example, it remains to study the injectivity of the Pr\"ufer modules associated to cycles. Suggested by notation used in \cite{AraRanga}, we give the following.  
\begin{definition}\label{defi:maximalcycle}
{\rm Let $E$ be a finite directed graph. A cycle $c$ based at $s(c)=v$  is said to be \emph{maximal} if there are no cycles in $E$ other than cyclic shifts of  $c$ which connect to $v$.}
\end{definition}
%
%{\bf XXX We should be clearer on what a cycle ``different from $c$" means.  We can include that language in the Section where we give the definitions.}
%\textcolor{red}{This is the language we used in the introduction to LPA in Section 3}
%
%%
%
In particular any source cycle is maximal.
We are now in position to state the main result of the article, 
which  characterizes when the Pr\"ufer module $\Uni E{c-1}$ is injective solely in terms of how the cycle $c$ sits in the graph $E$.

\begin{theorem}\label{thm:injective}
Let $E$ be a finite graph  and let $c$ be a basic closed path in $E$. Let $\Uni E{c-1}$ be the Pr\"ufer module associated to $c$.
Then $\Uni E{c-1}$ is injective if and only if $c$ is a maximal cycle.
%\begin{enumerate}
%\item If $c$ is a  cycle such that $\alpha c\neq 0$  for at most a finite number of distinct  paths $\alpha \in A_c$, then $\Uni E{c-1}$ is injective.  In particular, if $c$ is a source cycle, then $\Uni E{c-1}$ is injective.
%\item If $c$ is a  cycle such that $\alpha c\neq 0$  for  infinitely many  distinct  paths $\alpha \in A_c$, then $\Uni E{c-1}$ is not injective.
%\end{enumerate}

In case $\Uni E{c-1}$ is injective, then 

(1)  $\Uni E{c-1}$ is the injective envelope of the Chen simple module $V_{[c^{\infty}]}$, and 

(2)  ${\rm End}_{L_K(E)}(\Uni E{c-1})$ is isomorphic to the ring $K[[x]]$ of formal power series in $x$.
\end{theorem}

The proof of one direction of Theorem \ref{thm:injective} has already been established:  if $c$ is not a maximal cycle then $\Uni E{c-1}$ is not injective by Proposition~\ref{Ucnotpureinjective} (see also Example~\ref{VcinftyinR2notinjective}(2)). Establishing the converse implication will be a more difficult task, and will take up the remainder of this article.   
%We first reduce ourselves 
The strategy is to start by reducing  to the case when $c$ is a source loop, and then subsequently prove the result in this somewhat more manageable configuration.

%%%%%%%%%%%%%%%%%%%%%%%%%%%%%%%%%%%%%%%%%%%%%%%%%%%%%%
%%%%%%%%%%%%%%%%%%%%%%%%%%%%%%%%%%%%%%%%%%%%%%%%%%%%%%
%%%%%%%%%%%%%%%%%%%%%%%%%%%%%%%%%%%%%%%%%%%%%%%%%%%%%%
%%%%%%%%%%%%%%%%%%%%%%%%%%%%%%%%%%%%%%%%%%%%%%%%%%%%%%
% 
%   When c is a source loop section
%
%%%%%%%%%%%%%%%%%%%%%%%%%%%%%%%%%%%%%%%%%%%%%%%%%%%%%%
%%%%%%%%%%%%%%%%%%%%%%%%%%%%%%%%%%%%%%%%%%%%%%%%%%%%%%
%%%%%%%%%%%%%%%%%%%%%%%%%%%%%%%%%%%%%%%%%%%%%%%%%%%%%%

\section{Reduction from the general case to the source loop case}\label{Sectiongeneraltosourcecycle}

%\section{When $c$ is a  cycle with finitely many distinct paths entering on it}

We  assume now that  $L_K(E)$ is the Leavitt path algebra of a finite graph $E$ which contains a maximal cycle $c$ based at $v$. Then, as noted in Remark~\ref{rem:Gincasecissource},  $A_c$ is a finite set.  
We show that  we can reduce to the case where $c$ is a source cycle (i.e., $c$ is a cycle for which $A_c = \emptyset$).  

%\bigskip
%
%{\bf A/F:   Maybe we should make this clearer near the beginning of the article?  }
%
%\bigskip
%
%
%
%{\bf A/F:   Maybe we can add some comment here regarding the Prufer module for $K[x,x^{-1}]$?}
%
%\bigskip

% and we reduce to the case in which $c$ is a source cycle.
%   We claim that the $c^{\infty}$-Pr\"ufer module is injective.
Let $z\in E^0$ be a source vertex which is the source of a path entering on the cycle $c$; set $\varepsilon:=1-z$.  By \cite[Lemma 4.3]{AraRanga}, the Leavitt path algebras $L_K(E)$ and $S =\varepsilon L_K(E)\varepsilon \cong  L_K(E\setminus z)$ are Morita equivalent.  Note that $c$ is a cycle in the graph $E\setminus z$.
%, and in particular $c-1 := c- 1_{L_K(E\setminus z)}$ is in $S$.    
Since \begin{enumerate}
\item $c-1$ is neither  a right zero divisor nor left invertible in $L_K(E \setminus z)$, 
\item $(c-1)\varepsilon=\varepsilon(c-1)$, and
\item $(c-1)(1-\varepsilon)=-(1-\varepsilon)$,
\end{enumerate}
we can apply Proposition~\ref{prop:Moritaeq} to yield 
 % the uniserial left $L_K(E)$-module $\Muni nE{c-1}$ corresponds in the Morita equivalence to $\Muni nF{c-1}$. Since Morita equivalence respects direct limits, 
 that the Pr\"ufer $L_K(E)$-module $\Uni E{c-1}=\varinjlim_n \Muni nE{c-1}$ corresponds under the equivalence  to the Pr\"ufer $L_K({E\setminus z})$-module $\Uni {E\setminus z}{c-1}=\varinjlim_n \Muni n{E\setminus z}{c-1}$.   Moreover,  the original Pr\"ufer $L_K(E)$-module $\Uni E{c-1}$ is the injective envelope of the Chen simple $L_K(E)$-module $V_{[c^{\infty}]}$
if and only if the Pr\"ufer $L_K({E\setminus z})$-module $\Uni {E\setminus z}{c-1}$ is the injective envelope of the Chen simple $L_K({E\setminus z})$-module $V_{[c^{\infty}]}$.    

Thus by means of a finite number of ''source eliminations" we then may reduce $E$ to a subgraph which contains $c$, and in which $c$ is a source cycle, for which the Pr\"ufer modules correspond. 

The second step is to show that we can indeed further reduce to the case in which $c$ is a source loop. Assume $L_K(E)$ is a Leavitt path algebra with a source cycle $c$ based on the  vertex $v$. Assume $c$ has length $>1$ (i.e., that $c$ is not a source loop). Let $v:=v_1, v_2,..., v_n$ be the vertices of the cycle $c$ and $U=E^{0}\setminus \{v_2,..., v_n\}$. Consider the idempotent $\varepsilon:=\sum_{u\in U} u$. As proved in \cite[Lemma 4.4]{AraRanga}, $L_K(E)\varepsilon L_K(E)=L_K(E)$ and therefore $L_K(E)$ is Morita equivalent to $S:=\varepsilon L_K(E)\varepsilon$. 
%Let $F$ be the graph $(F^0,F^1)$ defined by:
%\begin{itemize}
%\item $F^0=E^0\setminus\{v_2,...,v_n\}$;
%\item $s_F^{-1}(w)=s_E^{-1}(w)$ for each $w\not=v$;
%\item $s_F^{-1}(v)= \{d\} \cup\displaystyle\bigcup_{i=1}^n\{f_g:g\in s_E^{-1}(v_i), r(g)\notin\{v_1,...,v_n\}\}$ where $d$ is a loop with $r(d)=v$ and the $f_g$'s are new edges with $r(f_g)=r(g)$.
%\end{itemize}
%Then, as described in \cite{AraRanga}, the map $\theta:L_K(F)\to L_K(E)$ defined by
%\begin{itemize}
%\item $\theta(w)=w$ for each $w\in F^0$,
%\item $\theta(e)=e$ for all $e$ with $s(e)\in F^0\setminus\{v\}$,
%\item $\theta(f_g)=e_1\cdots e_{i-1}g$ for each $g\in s_E^{-1}(v_i)$,
%\item $\theta(d)=e_1\cdots e_n=c$,
%\end{itemize}
%defines an isomorphism between $L_K(F)$ and the corner $S=\varepsilon L_K(E)\varepsilon$. Therefore let us prove that $\Muni nS{\varepsilon(c-1)}$ as $L_K(F)$-module is isomorphic to $\Muni nF{(d-1)}$.
Since 
\begin{enumerate}
\item $c-1$ is neither  a right zero divisor nor left invertible,
\item $(c-1)\varepsilon=\varepsilon(c-1)$, and 
\item $(c-1)(1-\varepsilon)=-(1-\varepsilon)$,
\end{enumerate}
by Proposition~\ref{prop:Moritaeq} the uniserial left $L_K(E)$-module $\Muni nE{c-1}$ corresponds in the Morita equivalence to $\Muni nS{\varepsilon(c-1)}$.

Let $F$ be the graph $(F^0,F^1)$ defined by:
\begin{itemize}
\item $F^0=E^0\setminus\{v_2,...,v_n\}$;
\item $s_F^{-1}(w)=s_E^{-1}(w)$ for each $w\not=v$;
\item $s_F^{-1}(v)= \{d\} \cup\displaystyle\bigcup_{i=1}^n\{f_g:g\in s_E^{-1}(v_i), r(g)\notin\{v_1,...,v_n\}\}$ where $d$ is a loop with $r(d)=v$ and the $f_g$'s are new edges with $r(f_g)=r(g)$.
\end{itemize}
Then, as described in \cite{AraRanga}, the map $\theta:L_K(F)\to L_K(E)$ defined by
\begin{itemize}
\item $\theta(w)=w$ for each $w\in F^0$,
\item $\theta(e)=e$ for all $e$ with $s(e)\in F^0\setminus\{v\}$,
\item $\theta(f_g)=e_1\cdots e_{i-1}g$ for each $g\in s_E^{-1}(v_i)$,
\item $\theta(d)=e_1\cdots e_n=c$,
\end{itemize}
defines an isomorphism between $L_K(F)$ and the corner $S=\varepsilon L_K(E)\varepsilon$. 

We now show that the left $L_K(F)$-modules $\Muni nS{\varepsilon(c-1)}$ and $\Muni nF{(d-1)}$ are isomorphic. Indeed, by Remark~\ref{rem:Gincasecissource} and Proposition~\ref{prop:Grep}, any element $x$ in $\Muni nE{c-1}$ can be written in a unique way as
\[x=g_1+g_2(c-1)+\cdots+g_n(c-1)^{n-1}+L_K(E)(c-1)^n,\]
with $g_j=k_j1_{L_K(E)}+t_{j,1}$ where $k_i\in K$ and $t_{j,1}$ is a $K$-linear combination of the paths $e_2\cdots e_n,..., e_{n-1}e_n, e_n$. Therefore, since $\varepsilon e_ie_{i+1}\cdots e_n=0$ for each $i>1$, the elements of $\Hom_{L_K(E)}(L_K(E)\varepsilon, \Muni nE{c-1})=\varepsilon \Muni nE{c-1}\cong \Muni nS{\varepsilon(c-1)}$ are of the type 
\[k_1\varepsilon+k_2\varepsilon(c-1)+\cdots+ k_n\varepsilon(c-1)^{n-1}+\varepsilon L_K(E)(c-1)^n\]
with $k_1,...,k_n\in K$. 
%Since $c$ is a source cycle in $L_K(E)$, we have $\varepsilon L_K(E)(c-1)^n=\varepsilon L_K(E)\varepsilon (c-1)^n$ and hence 
%\[k_1\varepsilon+k_2\varepsilon(c-1)+\cdots+ k_n\varepsilon(c-1)^n+\varepsilon L_K(E)(c-1)^n=k_1\varepsilon+k_2\varepsilon(c-1)+\cdots+ k_n\varepsilon(c-1)^n+S(c-1)^n.\]
%Era FALSO: \varepsilon e_1(c-1)^n\in \varepsilon L_K(E)(c-1)^n ma non a \varepsilon L_K(E)\varepsilon (c-1)^n
Since  
$$k_1\varepsilon+k_2\varepsilon(c-1)+\cdots+ k_n\varepsilon(c-1)^n=\theta\big(k_11_{L_K(F)}+k_2(d-1_{L_K(F)})+\cdots+ k_n(d-1_{L_K(F)})^n\big),$$
the $L_K(F)$-module $\Muni nS{\varepsilon(c-1)}$ coincides with $\Muni nF{d-1}$.
Since Morita equivalence respects direct limits, the Pr\"ufer module $\Uni E{c-1}=\varinjlim_n \Muni nE{c-1}$ corresponds to the Pr\"ufer module $\Uni F{d-1}=\varinjlim_n \Muni nF{d-1}$. 
Moreover, the Pr\"ufer $L_K(E)$-module $\Uni E{c-1}$ is the injective envelope of the Chen simple $L_K(E)$-module $V_{[c^{\infty}]}$
if and only if the Pr\"ufer $L_K(F)$-module $\Uni F{d-1}$ is the injective envelope of the Chen simple $L_K(F)$-module $V_{[d^{\infty}]}$.

Finally, since corresponding modules in a Morita equivalence have the same endomorphism ring, summarizing the discussion of this section, we have obtained the following.

\begin{proposition}\label{reducetosourceloop}
In order to establish Theorem \ref{thm:injective}, it suffices to prove that, whenever $c$ is a source loop in $E$, then 

(1)  $U_{E,c-1}$ is injective, and 

(2)  ${\rm End}_{L_K(E)}(U_{E,c-1}) \cong K[[x]]$.

\end{proposition}

%%%%%%%%%%%%%%%%%%%%%%%%%%%%
%
%
%%%%%\section{When $c$ is a source loop}
%
%
%%%%%%%%%%%%%%%%%%%%%%%%%%%%%

\medskip

\section{Establishing the main result:  the case when $c$ is a source loop}\label{Sectionsourcecycletosourceloop}
Having in the previous section reduced the verification of Theorem \ref{thm:injective} to the case where $c$ is a source loop, our aim in this section is to establish precisely that.   

 So suppose  $E$ is a graph in which there is a source loop $c$   based at the  vertex $v = s(c)$.
%Consider the infinite rational path $c^{\infty}$. Notice that i
In this case the  Chen simple module $V_{[c^{\infty}]}$  has $K$-dimension $1$, i.e $V_{[c^{\infty}]}=\{kc^{\infty} \ | \ k\in K\}$; moreover $A_c = \emptyset$, 
%\[A_c=\{\alpha\in\text{Path}(E): \alpha\text{ has length }\geq 1, \beta c\not=\alpha\not=c\beta\quad\text{for any real path $\beta$}, \text{ and }\alpha c\not=0\}=\emptyset,\]
and hence $G$ is the $K$-vector subspace of $L_K(E)$ generated by $1_{L_K(E)}$  (recall Definition~\ref{AcDefinition}).  
%Therefore  by Proposition \ref{prop:Grep},   any element in the  Leavitt path algebra $L_K(E)$ can be written, for a fixed $n\in\mathbb N$, in a unique way as
%\[h_1+h_2(c-1)+\cdots+h_{n}(c-1)^{n-1}+q_n(c-1)^n\]
%with $h_i\in K$ for $i=1,...,n$ and $q_n\in L_K(E)$.
By Proposition~\ref{prop:Grep} every  element of $\Muni nE{c-1}$ can be written 
in a unique way as
$$x = k_1+k_2(c-1)+\cdots+k_n(c-1)^{n-1}+L_K(E)(c-1)^n$$ with the $k_i$'s belonging to $K$. Therefore the elements of $\Uni E{c-1}$ can be written 
in a unique way as $K$-linear combinations of the $\alpha_i=\psi_{E,i}\big(1+L_K(E)(c-1)^i\big)$, $i\geq 1$.

Intuitively, the reason that reduction to the source loop case will provide a more manageable situation than the general case is  because the coefficients on each of the $(c-1)^i$ terms in the previous display come from $K$ (since $G = K$ in this case), and as such these coefficients  are central in $L_K(E)$.  
\begin{proposition}\label{prop:J}
Let $c$ be a source loop in $E$.   Then %we have 
%the following equalities of two-sided ideals.  
\begin{enumerate}
\item $L_K(E)(c-1)^n$ is the two-sided ideal $\Ann_{L_K(E)}(\Muni nE{c-1})$.
\item The left $L_K(E)$-module $\Muni nE{c-1}$ is also a right $L_K(E)$-module,  and
\[rm=mr\quad \forall r\in L_K(E), m\in \Muni nE{c-1}.\]
Thus the maps
$\psi_{E,i,j}:\Muni iE{c-1} \to \Muni jE{c-1}$ are also right $L_K(E)$-module monomorphisms for any $1\leq i\leq j$.
\item \textcolor{black}{$\Ann_{L_K(E)}(\Uni E{c-1})=\bigcap_{n\geq 1}L_K(E)(c-1)^n $.  and it coincides with the two-sided ideal $ \langle E^0\setminus \{s(c)\} \rangle $;} 
%\item $u\in \Uni E{c-1}$ belongs to $L_K(E)\alpha_i$ if and only if $(c-1)^iu=0$.
\item  The left $L_K(E)$-module $\Uni E{c-1}$ is also a right $L_K(E)$-module and $\psi_{E,n}:\Muni nE{c-1}\to \Uni E{c-1}$ is  a right $L_K(E)$-module monomorphism. Moreover $r\alpha_i=\alpha_ir$ for each $r\in L_K(E)$ and  $i\geq 1$.
\item $u\in \Uni E{c-1}$ belongs to $L_K(E)\alpha_i$ if and only if $(c-1)^iu=0$.
%
%
%\item $\Ann_{L_K(E)}(\Uni E{c-1})=\bigcap_{n\geq 1}L_K(E)(c-1)^n = \langle E^0\setminus \{s(c)\} \rangle $ (the two-sided ideal of $L_K(E)$ generated by $E^0\setminus \{s(c)\}$);   
%\item $u\in \Uni E{c-1}$ belongs to $L_K(E)\alpha_i$ if and only if $(c-1)^iu=0$;
%\item $\Muni nE{c-1}$ is an $L_K(E)$-bimodule and the maps
%$\psi_{E,i,j}:\Muni iE{c-1} \to \Muni jE{c-1}$ are $L_K(E)$-bimodule monomorphisms for any $1\leq i\leq j$;
%\item $\Uni E{c-1}$ is an $L_K(E)$-bimodule and $\psi_{E,n}:\Muni nE{c-1}\to \Uni E{c-1}$ is an $L_K(E)$-bimodule monomomorphism;
%%\item $\Ann_{L_K(E)}(\Muni nE{c-1})=L_K(E)(c-1)^n$.  
%\item $(c-1)\alpha_i=\alpha_i(c-1)$ for each $i\geq 1$.
%%   Moreover, this ideal  is generated as two sided ideal by $E^0\setminus \{s(c)\}$.
\end{enumerate}
%Moreover $u\in \Uni E{c-1}$ belongs to $L_K(E)\alpha_i$ if and only if $(c-1)^iu=0$.
\end{proposition}
%
%\bigskip
%
%{\bf A/F:  maybe we can put some statement about why we expect this to happen, it is due to $c-1$ commuting with elements of $K$.}
%
%\bigskip
%
\begin{proof}
(1)   If $r\in \Ann_{L_K(E)}(\Muni nE{c-1})$, then $r(1+L_K(E)(c-1)^n)=0$ in $\Muni nE{c-1}$ and hence $r$ belongs to $L_K(E)(c-1)^n$. Conversely, let $r\in L_K(E)$ and $m\in \Muni nE{c-1}$. Since $m=h_1+h_2(c-1)+\cdots+h_n(c-1)^{n-1}+L_K(E)(c-1)^n$ where each $h_i \in K$ (using that $c$ is a source loop; see the previous observation), we get 
% {\bf XXX because we have a source loop}   we get
\begin{align*} 
r(c-1)^nm & =r(c-1)^n(h_1+h_2(c-1)+\cdots+h_n(c-1)^{n-1}+L_K(E)(c-1)^n) \\
 & =   h_1r(c-1)^n    +h_2r(c-1)^{n+1} +\cdots+h_nr(c-1)^{2n-1}+L_K(E)(c-1)^n \\
 & =     0
\end{align*}
in $\Muni nE{c-1}=L_K(E)/L_K(E)(c-1)^n$.  (The point here is that each $h_i$ commutes with expressions of the form $r(c-1)^j$ because $h_i\in K$.)  
% {\bf XXX We note that the $h_i$ commute with expressions of the form $r(c-1)^i$ because $h_i \in K$; this is where we are using that $c$ is a source LOOP.}     
Hence $L_K(E)(c-1)^n\leq \Ann_{L_K(E)}(\Muni nE{c-1})$.

\smallskip

(2)  Since $L_K(E)(c-1)^n$ is a two-sided ideal by point (1), then $\Muni n E{c-1}$ is also a right $L_K(E)$-module via the usual action. Let $r\in L_K(E)$ and $m\in \Muni n E{c-1}$; then
\[r=k_1+k_2(c-1)+\cdots+k_n(c-1)^{n-1}+r'(c-1)^n \ \ \ \mbox{and} \ \ \ 
m=h_1+\cdots+h_n(c-1)^{n-1}+L_K(E)(c-1)^n,\]
where $h_1,..., h_n,k_1,..., k_n\in K$ and $r'\in L_K(E)$. Since 
$L_K(E)(c-1)^n$ is a two-sided ideal we get $rm=mr$. The right $L_K(E)$-linearity of the maps $\psi_{E,i,j}$ for each $1\leq i\leq j$ follows easily.

\smallskip

(3)   Since $\Uni E{c-1}=\bigcup_{n\geq 1}L_K(E)\alpha_n$ and $L_K(E)\alpha_n\cong\Muni nE{c-1}$, the first equality follows from (1). \\   
For the second, we start by noting that $E^0\setminus \{s(c)\}$ is the set of the vertices contained in $\Ann_{L_K(E)}(\Uni E{c-1})$.  Indeed, $s(c)=1+(1-s(c))(c-1)$ does not belong to $L_K(E)(c-1)$, and hence neither to $\Ann_{L_K(E)}(\Uni E{c-1})$.  On the other hand, any vertex $w\not=s(c)$ belongs to $\bigcap_{n\geq 1}L_K(E)(c-1)^n$, because  the equality $w=-w(c-1)$ can be iterated to produce the sequence 
\[w=-w(c-1)=w(c-1)^2=\cdots=(-1)^nw(c-1)^n=\cdots .\]
In \cite[Theorem 4]{Ranga}, Rangaswamy proved that an arbitrary nonzero two sided ideal $I$ in $L_K(E)$ (for  $E$ a finite graph)  is generated by the union of two sets:
\begin{enumerate}
\item[(i)] $I\cap E^0$ (i.e., the vertices in $I$), together with
\item[(ii)] a (possibly empty) set of mutually orthogonal elements of the form $u+\sum_{i=1}^n k_ig^i$ where $u \in E^0\setminus I\cap E^0$, $k_1,...,k_n$ belong to $K$ with $k_n\not=0$, and $g$ is a cycle without exits in $E^0\setminus I\cap E^0$ based at the vertex $u$.
\end{enumerate}
In our case we have
\begin{itemize}
\item $\Ann_{L_K(E)}(\Uni E{c-1})\cap E^0=E^0\setminus\{s(c)\}$, and 
\item $c$ is a cycle in $E^0 \setminus (I \cap E^0) = \{s(c)\}$, and is the only cycle in the only cycle based in $s(c)$ (because $c$ is a source loop), and $c$ has no exits in $\{s(c)\}$ (because such an exit in $\{s(c)\}$ would necessarily be a second loop at $s(c)$, contrary to $c$ being a source loop).  
% without exits in $s(c)$.
\end{itemize}
Therefore $\Ann_{L_K(E)}(\Uni E{c-1})$ is generated by $E^0\setminus\{s(c)\}$ and possibly a single element  of the form $s(c)+\sum_{i=1}^n k_ic^i$ with $k_n\not=0$.
%So the desired result will follow by  proving that $\Ann_{L_K(E)}(\Uni E{c-1})$ does not contain any element of the form $s(c)+\sum_{i=1}^n k_ic^i$, $k_n\not=0$. 
Assume that $s(c)+\sum_{i=1}^n k_ic^i \in \Ann_{L_K(E)}(\Uni E{c-1})$ where $k_n \neq 0$.   
We have
$$s(c)=1+(-1)^{n-1}(1-s(c))(c-1)^n$$ and, by applying  Lemma~\ref{rem:c^n} to each $c^i$ and then collecting like powers of $c-1$, we see 
\[\sum_{i=1}^n k_ic^i= \sum_{i=1}^n k_i+  ( \sum_{i=1}^n \binom i{1} k_i ) (c-1)+\cdots+ ( \sum_{i=j}^n \binom i{j} k_i) (c-1)^j+\cdots+k_n(c-1)^n.\]
Therefore, using the displayed equation (and separating out the leading $1$ term), we get that  $s(c)+\sum_{i=1}^n k_ic^i$ is equal to
\[1+ \sum_{i=1}^n k_i+\sum_{i=1}^n \binom i{1} k_i(c-1)+\cdots+\sum_{i=j}^n \binom i{j} k_i(c-1)^j+\cdots+k_n(c-1)^n+(-1)^{n-1}(1-s(c))(c-1)^n.\]
Since $\big(1-s(c)\big)(c-1)=-\big(1-s(c)\big)$, the final summand $(-1)^{n-1}(1-s(c))(c-1)^n$ coincides with $(-1)^{m+n-1}(1-s(c))(c-1)^{m+n}$ for each $m\geq 0$, and so it belongs to $\Ann_{L_K(E)}(\Uni E{c-1})$. Therefore, the element $s(c)+\sum_{i=1}^n k_ic^i$ belongs to $\Ann_{L_K(E)}(\Uni E{c-1})$ if and only if 
\[1+ \sum_{i=1}^n k_i+\sum_{i=1}^n \binom i{1} k_i(c-1)+\cdots+\sum_{i=j}^n \binom i{j} k_i(c-1)^j+\cdots+k_n(c-1)^n\]
belongs to $\Ann_{L_K(E)}(\Uni E{c-1})$.  \textcolor{black}{In such a situation,  the displayed element must  annihilate in particular the elements $\alpha_1$, ..., $\alpha_n$. By successively multiplying this equation in turn by $\alpha_1, \alpha_2, \dots, \alpha_n$, and using the displayed observation made prior to Remark \ref{remark:levels}, we get that 
\[0=1+ \sum_{i=1}^n k_i=\cdots=\sum_{i=j}^n \binom i{j} k_i=\cdots=k_n,
\]}
which contradicts that  $k_n\not=0$.

\smallskip

(4) The first claim follows immediately by point (2). Moreover
\[r\alpha_i=r\psi_{E,i}(1+L_K(E)(c-1)^i)=\psi_{E,i}(r+L_K(E)(c-1)^i)=\psi_{E,i}(1+L_K(E)(c-1)^i)r=\alpha_ir\]
for each $r\in L_K(E)$ and $i\geq 1$.

\smallskip

(5) Any $u\in \Uni E{c-1}$ can be written as
\[u=k_1\alpha_1+\cdots+k_n\alpha_n\]
for a suitable $n\geq 1$. Since $\Muni jE{c-1}\cong L_K(E)\alpha_j$, we have
\[(c-1)^iu=0\quad\forall i\geq n\]
and, if $i<n$, then 
\[(c-1)^iu=k_{i+1}(c-1)^i\alpha_{i+1}+\cdots+k_n(c-1)^i\alpha_n=k_{i+1}\alpha_{1}+\cdots+k_n\alpha_{n-i}.\]
Therefore $(c-1)^iu=0$ if and only if $k_{i+1}=\cdots=k_n=0$ if and only if $u\in L_K(E)\alpha_i$.
\end{proof}

\begin{remark}\label{Jgraded}
We note that although each two-sided ideal $L_K(E)(c-1)^n$ is not graded (because it contains neither $c$ nor $1$), the intersection $J=\bigcap_{n\in \mathbb N} L_K(E)(c-1)^n$ is graded (because it has been shown to be generated as a two-sided ideal by a set of vertices).  
\end{remark}

%\begin{remark}\label{Jnottwosidedfornotloop}
%In general if $c$ is not a loop, but is only assumed to be a source cycle, the left ideal $L_K(E)(c-1)^n$ is not a right ideal. 
%
%\bigskip
%
%{\bf A/F:  is this NEVER a right ideal for any non-loop?}
%
%\bigskip
%
% For example, let $c = ef $ be a cycle of the form 
%$$
%\xymatrix{
%  \bullet^v \ar@/^/[r]^{e}
%& \bullet^w  \ar@/^/[l]^{f} 
%%& \bullet^{} \ar@(ul,ur)[] \ar@/^/[r] \ar@/^/[l]
%%& \bullet^{} \ar@(ur,dr)[] \ar@/^/[l]
%}
%$$
%in a graph $E$.  By the previously-cited \cite[Proposition 2.6]{AMT1}, $L_K(E)(c-1) = {\rm Ker}(\hat{\rho}_{c^\infty})$.  Note that $fc \neq 0$ but $cf = 0$ in $L_K(E)$.   Now  $(c-1)f c^\infty = cfc^\infty - fc^\infty = 0 - fc^\infty \neq 0$ in $V_{[c^\infty]}$, so that $(c-1)f \notin {\rm Ker}(\hat{\rho}_{c^\infty})$, so $(c-1)f \notin L_K(E)(c-1)$, which gives that $L_K(E)(c-1)$ is not a right ideal. 
%Nevertheless, if $c=e_1\cdots e_n$ and we set $c_i:=e_i\cdots e_n e_1\cdots e_{i-1}$, one can prove, analogously to Proposition~\ref{prop:J}, that  $\Ann_{L_K(E)}(\Muni mR{c-1}) = \bigcap_{i=1}^nL_K(E)(c_i-1)^m$ and $\Ann_{L_K(E)}(\Uni E{c-1}) = \bigcap_{m\in\mathbb N}\Big(\bigcap_{i=1}^nL_K(E)(c_i-1)^m\Big)$.
%\end{remark}

\begin{proposition}\label{C*annih}  
Let $c$ be a source loop.   For any $j\in \Ann_{L_K(E)}(\Uni E{c-1})$
there exists $n \in \N$ such that ${c^*}^n j=0$.
\end{proposition}
\begin{proof}
By Proposition \ref{prop:J}(3), any nonzero $j \in \Ann_{L_K(E)}(\Uni E{c-1})$ is a $K$-linear combination of elements of the form $\alpha\beta^*w\gamma\delta^*\not=0$, with $\alpha$, $\beta$, $\gamma$ and $\delta$ real paths and $w\not=s(c)$ a vertex in $E$. Let us concentrate on one of these elements. If $\alpha\beta^*w=w$ then $c^*\alpha\beta^*w\gamma\delta^*=c^*w\gamma\delta^*=0$. If 
$\alpha\beta^*w=\beta^*w\not=w$ then $s(\beta^*)=r(\beta)\not=s(c)$, otherwise $\beta$ would be a path which starts in $w\neq s(c)$ and ends at $s(c)$, contrary to $c$ being a source loop; then $c^*\alpha\beta^*w\gamma\delta^*=c^*\beta^*w\gamma\delta^*=0$.   
%
%
%
%\bigskip
%
%{\bf A/F:  we should highlight this more; this is where we need that $c$ is a loop, not a general cycle.  Is there anywhere else we need this?}
%
%\bigskip
%
In all the other cases $\alpha = c^t\eta_1\cdots\eta_s$ with $c\not=\eta_1\in E^1$, $t\geq 0$ and $s\geq 1$. Then
\[(c^{t+1})^*\alpha\beta^*w\gamma\delta^*=(c^{t+1})^*c^t\eta_1\cdots\eta_s\beta^*w\gamma\delta^*=c^*\eta_1\cdots\eta_s\beta^*w\gamma\delta^*=0.\]
Since $j$ is a finite sum of terms of the form $\alpha\beta^*w\gamma\delta^*$, we achieve the desired conclusion.   
\end{proof}

%\bigskip
%
%{\bf  A/F:    I have not yet looked carefully at the following result and proof.}
%
%\bigskip

%QUIIIII
\begin{proposition}\label{prop:baer}
For any $\ell \in L_K(E)\setminus \Ann_{L_K(E)}(\Uni E{c-1})$ and for any $u\in \Uni E{c-1}$, there exists $X\in \Uni E{c-1}$ such that $\ell X=u$. That is, $u$  is divisible by any element   in $L_K(E)\setminus \Ann_{L_K(E)}(\Uni E{c-1})$.
\end{proposition}
%
%\bigskip
%
%{\bf I have not read the following proof carefully yet.}
%
%\bigskip
%
%
\begin{proof}
% As usual, denote by $M_n$ the uniserial submodule of $U$ of length $n$. Moreover, since $A_c=\emptyset$, we have that the element of the uniserial module $M_n=R/R(c-1)^n$ can be written in a unique way  (HERE we have to be more precise, since the elements are in the quotient modulo $R(c-1)^n$)
% \[k_1+k_2(c-1)+\cdots +k_n(c-1)^{n-1}+R(c-1)^n%=(k_1,...,k_{n-1},k_n)
% \]
%where $k_i\in K$. 
Let us consider $u\in \Uni E{c-1}$. Then, \textcolor{black}{as observed at the beginning of this section}, we have 
\[u=k_1\alpha_n+k_2\alpha_{n-1}+\cdots +k_n\alpha_1\]
where $k_i\in K$. Since $\ell\notin \Ann_{L_K(E)}(\Uni E{c-1})$, by Proposition~\ref{prop:J} there exists $m\in \mathbb N$ such that $\ell$ is not right-divisible by $(c-1)^m$.   Therefore
\[\ell=h_1+h_2(c-1)+\cdots+h_{m}(c-1)^{m-1}+q_m(c-1)^m\]
with $h_i\in K$ for $i=1,...,m$, $q_m\in L_K(E)$ and \[(h_1,...,h_{m})\not=(0,0,...,0).\]
Let $s$ be the minimum index such that $h_{s+1}\not=0$.
It is not restrictive to assume $m\geq n+s$: otherwise we apply the division algorithm to $q_m$, $q_{m+1}$, ... until we get
\[\ell=h_1+h_2(c-1)+\cdots+h_{m}(c-1)^{m-1}+\cdots+h_{n+s}(c-1)^{n+s-1}+q_{n+s}(c-1)^{n+s}.\]
%The element $u\in R\alpha_n$ belongs also to $R\alpha_{n+s}$; as element of $R\alpha_{n+s}$ it has the following aspect
%\[u=k_1(c-1)^{n-1}(c-1)^s+\cdots+k_{n-1}(c-1)(c-1)^s+k_n(c-1)^s=(k_1,...,k_{n-1},k_n,\underbrace{0,0,...,0}_s).\]
We claim that the equation $\ell X=u$ has solutions in $L_K(E)\alpha_{n+s}$, as follows.  Set $X=X_{1}\alpha_{n+s}+\cdots+X_{n+s-1}\alpha_2+X_{n+s}\alpha_1$.  
We solve \[\ell\big(X_{1}\alpha_{n+s}+\cdots+X_{n+s-1}\alpha_2+X_{n+s}\alpha_1
\big)=u,\]
that is
\[\big(h_1+\cdots+h_{m}(c-1)^{m-1}+q_m(c-1)^m\big)\big(X_{1}\alpha_{n+s}+\cdots+X_{n+s}\alpha_1
\big)=k_1\alpha_n+\cdots+k_n\alpha_1.\]
%
%(c-1)^{n-1}(c-1)^s+\cdots+k_{n-1}(c-1)(c-1)^s+k_n(c-1)^s.\]
This yields the following equations in the field $K$:
\[h_1X_{1}=0,\ ...,\ \sum_{i=1}^{s}h_iX_{s+1-i}=0,\  \sum_{i=1}^{s+1}h_iX_{s+2-i}=k_1,\ \sum_{i=1}^{s+2}h_iX_{s+3-i}=k_2,...,\  
\sum_{i=1}^{s+n}h_iX_{s+n+1-i}=k_n.\]
Since $0=h_1=\cdots=h_{s}$ we get
\[h_{s+1}X_{1}=k_1,\ h_{s+1}X_{2}+h_{s+2}X_{1}=k_{2},..., \sum_{i=s+1}^{s+n}h_{i}X_{s+n+1-i}=k_n\]
from which we obtain the values of $X_{1}$, ..., $X_{n}$. The values of $X_{n+1}$, ..., $X_{n+s}$ can be chosen arbitrarily.
%
%{\bf XXX I need to think more about this proof.   Is there again something special about the fact that $c$ is a source loop?}
%\textcolor{red}{Yes! The element $u\in U$ is a $K$-linear combination of the $\alpha_i$ because $c$ is a source loop!!!}
%
\end{proof}

\begin{corollary}\label{cstarunonzero}
If $0 \neq u \in \Uni E{c-1}$ then $(c^*)^m u \neq 0$ for all $m\in \N$.
\end{corollary}
\begin{proof}
Since $c\notin L_K(E)(c-1)\supseteq \Ann_{L_K(E)}(\Uni E{c-1})$, by Proposition \ref{prop:baer} there exists $0\not=x\in \Uni E{c-1}$ with $cx = u$. Since $cs(c) = c$ we may assume that $s(c)x = x$.  Then $0\neq x = s(c)x =c^*cx = c^*u$.  Repeating the same argument for  $0\not=c^*u\in \Uni E{c-1}$, we get $(c^*)^2u\neq 0$; iterating, we get the result.  
\end{proof}
%
%\begin{remark}
%Here we use the observation made in Remark~\ref{remark:levels}. The idea is: a certain equation has no solution at the $n$-level in the uniserial Pr\"ufer module $\Uni E{c-1}$, but it has a solution at a higher level. And this fact gives us the divisibility of the Pr\"ufer modules by all the elements in $L_K(E)\setminus \Ann_{L_K(E)}(\Uni E{c-1})$. 
%\end{remark}
%

\begin{proposition}\label{prop:div}
%Let $c$ be a source loop in $E$.   
Let $c$ be a source loop in $E$.   Let $I_f$ be a finitely generated  left ideal of $L_K(E)$, and let $\varphi: I_f\to \Uni E{c-1}$ be a $L_K(E)$-homomorphism. Then there exists $\psi: L_K(E)\to \Uni E{c-1}$ such that $\psi |_{I_f}=\varphi$.  Consequently, ${\rm Ext}^1(L_K(E)/I_f, \Uni E{c-1}) = 0$. 
\end{proposition}
\begin{proof}
It has been established in \cite{AMT2} that $L_K(E)$ is a B\'{e}zout ring, i.e., that every finitely generated left ideal of $L_K(E)$ is principal.   So  $I_f = L_K(E)\ell$ for some $\ell \in I_f$.  \\
%$\ell$ to be a generator of $I$. 
Assume on one hand  that $\ell \in \Ann_{L_K(E)}(\Uni E{c-1})$, and hence $I_f\leq \Ann_{L_K(E)}(\Uni E{c-1})$.
By Proposition \ref{C*annih}, any element of $\Ann_{L_K(E)}(\Uni E{c-1})$ is annihilated by a suitable ${c^*}^n$.   Further, ${c^*}^nu\neq 0$ for any $0\neq u\in \Uni E{c-1}$ by Corollary \ref{cstarunonzero}.  Thus in this case we must have  $\Hom_{L_K(E)}(I_f, \Uni E{c-1})=0$, so that $\varphi = 0$ and the conclusion follows trivially.  \\
Assume on the other hand  that $\ell \notin \Ann_{L_K(E)}(\Uni E{c-1})$.
%Let us prove that $\psi= \rho_{(\ell)\varphi}$, the right multiplication by $(\ell)\varphi$, extends $\varphi$ to $L_K(E)$. 
By Proposition \ref{prop:baer}, there exists $x\in \Uni E{c-1}$ for which $\ell x = \varphi(\ell)$.  Let $\psi:L_K(E)\to \Uni E{c-1}$ be the extension of the map defined by setting $\psi(1)=x$. Then, for each $i=r_i\ell\in I_f$, we have
\[\psi(i)=\psi(r_i\ell)=r_i\ell\psi(1)=r_i\ell x=r_i \varphi(\ell)=\varphi(r_i\ell)=\varphi(i),\]
%
%
% Then for each $i\in I_f$, 
%$$ (i)\varphi = (r \ell)\varphi = r (\ell) \varphi = r (\ell x) = (r \ell) x = r \ell (\ell)\varphi = i (\ell) \varphi = (i) \rho_{(\ell)\varphi} = (i)\psi,$$
%then the statement follows from the divisibility of $U_c$ by any element not in $J$ established in Proposition \ref{prop:baer}.   
which establishes the desired conclusion in this case as well. \\  
The final statement is then immediate.  
\end{proof}

A submodule $N$ of a module $M$ is \emph{pure} if for each finitely presented module $F$, the functor $\Hom(F,-)$ preserves exactness of the short exact sequence $0\to N\to M\to M/N\to 0$. Modules that are injective with respect to pure embeddings are called \emph{pure-injective}.  

By \cite[Lemma 4.2.8]{Prest}, a module which is linearly compact over its endomorphisms ring is pure-injective. We will prove that $\Uni E{c-1}$ is artinian over the ring $\End(\Uni E{c-1})$ and therefore linearly compact. 

\begin{proposition}\label{prop;endomorphism}
Let $E$ be a finite graph, and $c$ a source loop in $E$. Then the endomorphism ring of the left $L_K(E)$-module $\Uni E{c-1}$ is isomorphic to the ring of formal power series $K[[x]]$.
\end{proposition}
\begin{proof}
Let $\varphi\in \End(\Uni E{c-1})$. 
Since
\[(c-1)^i\varphi(\alpha_i)=\varphi\big((c-1)^i\alpha_i\big)=\varphi(0)=0,\]
by Proposition~\ref{prop:J} $\varphi(\alpha_i)$ belongs to $L_K(E)\alpha_i$. Therefore
\[\varphi(\alpha_i)=h_{1,i}\alpha_i+\cdots+h_{i,i}\alpha_1.\]
Since
\begin{align*}
h_{1,i}\alpha_i+\cdots+h_{i,i}\alpha_1 & =\varphi(\alpha_i)  = \varphi((c-1)\alpha_{i+1}) \\
& = (c-1)\varphi(\alpha_{i+1})= (c-1)(h_{1,i+1}\alpha_{i+1}+\cdots+h_{i,i+1}\alpha_2+h_{i+1,i+1}\alpha_{1})\\
& = h_{1,i+1}\alpha_i+\cdots +h_{i,i+1}\alpha_1,
\end{align*}
%\[h_{1,i}\alpha_i+\cdots+h_{i,i}\alpha_1=\varphi(\alpha_i)=\varphi((c-1)\alpha_{i+1})=\]
%\[=(c-1)\varphi(\alpha_{i+1})=
%(c-1)(h_{1,i+1}\alpha_{i+1}+\cdots+h_{i,i+1}\alpha_2+h_{i+1,i+1}\alpha_{1})=\]
%\[=h_{1,i+1}\alpha_i+\cdots +h_{i,i+1}\alpha_1,\]
we get $h_{j,i+1}=h_{j,i}=:h_j$ for each $1\leq j\leq i$.
Denote by $H_\varphi(x)$ the formal power series
\[\sum_{j=1}^\infty h_jx^{j-1}.\]
It is easy to check that the map $\varphi\mapsto H_\varphi(x)$
defines a ring  monomorphism $\Phi$ between $\End(\Uni E{c-1})$ and $K[[x]]$. Given any formal power series $H(x)=\displaystyle\sum_{j=1}^\infty h_jx^{j-1}$, setting
\[\varphi_H(\alpha_i)=h_1\alpha_i+\cdots+h_i\alpha_1\]
one defines an endomorphism of $\Uni E{c-1}$. Indeed
\begin{align*}
(c-1)\varphi_H(\alpha_{i+1})  & =(c-1)(h_1\alpha_{i+1}+\cdots+h_i \alpha_2 + h_{i+1}\alpha_1) \\
& =   \ h_1\alpha_{i}+\cdots+h_{i}\alpha_1 + h_{i+1} 0  \\
& =   \varphi_H(\alpha_{i}) \ = \ \varphi_H((c-1)\alpha_{i+1}).  \qedhere 
 \end{align*}
%{\bf XXX I changed the previous sentence, I think it was backwards, can you check that it is now correct?}\textcolor{red}{Yes, now it is correct}    Clearly, $\Phi(\varphi_H)=H(x)$ and hence $\Phi$ is a ring isomorphism.
\end{proof}

\begin{corollary}\label{cor:serieformali}
Any endomorphism of $\Uni E{c-1}$ is the right product by a formal power series $\displaystyle\sum_{j=1}^\infty h_j(c-1)^{j-1}$ with coefficients $h_j\in K$.
\end{corollary}
\begin{proof}
Following the notation of Proposition~\ref{prop;endomorphism},
the endomorphism $\varphi_H$ associated to the formal power series $H=\sum_{j=1}^\infty h_jx^{j-1}$ sends $\alpha_i$ to 
\[h_1\alpha_i+\cdots+h_i\alpha_1=(h_1+h_2(c-1)+\cdots+h_i(c-1)^{i-1})\alpha_i=
\alpha_i(h_1+h_2(c-1)+\cdots+h_i(c-1)^{i-1}),\]
where the latter equality follows by point (4) of Proposition~\ref{prop:J}.
Since $\alpha_i(c-1)^j=(c-1)^j\alpha_i=0$ for each $j\geq i$, we can define
\[\alpha_i\sum_{j=1}^\infty h_j(c-1)^{j-1}:=\alpha_i(h_1+h_2(c-1)+\cdots+h_i(c-1)^{i-1})=\varphi_H(\alpha_i).   \qedhere   \] 
\end{proof}

\begin{proposition}\label{Ucpureinjective} 
Let $E$ be a finite graph, and $c$ a source loop in $E$.
%Suppose that  $\alpha c\neq 0$  for at most a  finite number of distinct real paths $
%\alpha  \in A_c$. 
Then $\Uni E{c-1}$ is pure-injective.
Consequently, if $(N_\alpha, f_{\alpha, \beta})$ is a direct system of left $L_K(E)$-modules and $L_K(E)$-homomorphisms,  then $\Ext^1(\varinjlim N_\alpha, \Uni E{c-1})=\varprojlim \Ext^1(N_\alpha, \Uni E{c-1})$.  
\end{proposition}
\begin{proof}
The left $L_K(E)$-module $\Uni E{c-1}$ is the union of its submodules $\{L_K(E)\alpha_i \ | \ i\geq 1\}$. Let $T$ denote the endomorphism ring $\End (\Uni E{c-1})$. By Corollary~\ref{cor:serieformali}, $L_K(E)\alpha_i$ is a right $T$-submodule of $\Uni E{c-1}$ for each $i\geq 1$. We show that these are the unique right $T$-submodules of $\Uni E{c-1}$. 

If $N$ is a finitely generated $T$-submodule of $\Uni E{c-1}$, let $i_N$ be the smallest natural number $i$ such that $N\leq L_K(E)\alpha_i$. If $i_N=1$, $L_K(E)\alpha_{i_N}=L_K(E)\alpha_1$ is a one dimensional $K$-vector space and hence a simple $T$-module. If $i_N\geq 2$, consider $n\in N\setminus L_K(E)\alpha_{i_N-1}$. Then
\[n=k_1\alpha_{i_N}+\cdots+k_{i_N}\alpha_1=\alpha_{i_N}(k_1+\cdots+k_{i_N}(c-1)^{i_N-1})\]
with $k_1\not=0$. Again invoking Corollary~\ref{cor:serieformali}, let  $\sum_{j=1}^\infty h_j(c-1)^{j-1}$ be the inverse of $k_1+\cdots+k_{i_N}(c-1)^{i_N-1}$ in $K[[c-1]]$, which exists as $k_1\neq 0$. Then
\[n\sum_{j=1}^\infty h_j(c-1)^{j-1}=\alpha_{i_N}\]
and hence $N=L_K(E)\alpha_{i_N}$.
%{\bf XXX These last two should be $\alpha_{i_N}$, yes?}\textcolor{red}{YES}

If on the other hand $N$ is not finitely generated, write $N=\varinjlim N_{\lambda}$, where the $N_{\lambda}$ are the finitely generated right $T$-submodules of $N$. For any $\lambda$, by the previous paragraph, there exists $j_{\lambda}$ such that $N_{\lambda}=L_K(E)\alpha_{j_{\lambda}}$. Since $N\not=N_\lambda$ for any $\lambda$, the sequence $({j_{\lambda}})_\lambda$ is unbounded and so $N=\Uni E{c-1}$.

Thus, since $\{L_K(E)\alpha_i:i\geq 1\}$ has been shown to be  the lattice of the proper right $T$-submodules of $\Uni E{c-1}$, we conclude that $\Uni E{c-1}$ is an artinian right $T$-module and hence linearly compact. 
%***************************
%
%
% 
%By Remark~\ref{rem:Gincasecissource}
%%From the assumption it follows that $\alpha c\neq 0$  for at most a  finite number of distinct real paths $
%%\alpha  \in A_c$. Then
%%
%%First notice that, if  $\alpha c\neq 0$ in $L_K(E)$   for a finite number of distinct real paths $\alpha \in A_c$, then
%the  $K$-vector space $V_{[c^{\infty}]}$ has finite dimension. Set $T:=\End_{L_K(E)}(\Uni E{c-1})$. We show that $\Uni E{c-1}$ is an artinian right $T$-module. Since $_{L_K(E)}\Uni E{c-1}=\varinjlim_{i\in \N} \Muni iE{c-1}=\bigcup_{i\in \N}  L_K(E)\alpha_i$ and 
%$L_K(E)\alpha_n=\{m\in \Uni E{c-1}: (c-1)^n m=0\}$, we get that for any $\varphi\in T$ and $m\in L_K(E)\alpha_n$ one has
%\[0=(0)\varphi=((c-1)^n m)\varphi=(c-1)^n(m)\varphi,\]
%and hence $(m)\varphi$ belongs to $L_K(E)\alpha_n$.
% So any $L_K(E)\alpha_i$ is a $T$-submodule of $\Uni E{c-1}$ and $L_K(E)\alpha_i\leq L_K(E)\alpha_{i+1}$ also as $T$-modules.  Since $L_K(E)\alpha_{i+1}/L_K(E)\alpha_i\cong V_{[c^{\infty}]}$ is a $K$-vector space of finite dimension, it is  of finite length  also as a $T$-module.  Hence $\Uni E{c-1}=\bigcup L_K(E)\alpha_i$, where the $L_K(E)\alpha_i$ are  $T$-modules of finite length, and  $\Uni E{c-1}$ is a uniserial right $T$-module for which  the $L_K(E)\alpha_i$'s are the unique $T$-submodules.  Thus $\Uni E{c-1}$ is an artinian right $T$-module,  and so is linearly compact. 
 By \cite[Lemma 4.2.8]{Prest}  
 %[M. Prest, \emph{Purity, Spectra and Localisation}, Lemma 4.2.8], 
 we get that the left $L_K(E)$-module $\Uni E{c-1}$ is pure-injective.  (The quoted result says: If a module is linearly compact over its endomorphism ring, then it is algebraically compact and hence pure-injective.)
 Therefore we may invoke   \cite[Lemma 3.3.4]{GT} to conclude that  the functor $\Ext^1(-, \Uni E{c-1})$ sends direct limits  to inverse  limits.
\end{proof}

\begin{proposition}\label{sourceloopresult}
Let $E$ be a finite graph with source loop $c$.   Then the Pr\"ufer module $\Uni E{c-1}$ is injective.  Indeed, $\Uni E{c-1}$  is the injective envelope of $L_K(E)\alpha_1\cong V_{[c^{\infty}]}$.
\end{proposition}
\begin{proof}
In order to check the injectivity of $\Uni E{c-1}$, we apply  Baer's Lemma. We need only check that $\Uni E{c-1}$ is injective relative to any short exact sequence of the form  $0\to I\to L_K(E)\to L_K(E)/I\to 0$. This  is equivalent to showing that $\Ext^1_{L_K(E)}(L_K(E)/I, \Uni E{c-1})=0$ for any left ideal $I$ of $L_K(E)$. Write  $I=\varinjlim I_{\lambda}$, where the $I_{\lambda}$ are the finitely generated submodules of $I$.  It is standard that $L_K(E)/I=\varinjlim L_K(E)/{I_{\lambda}}$.  So now applying the functor $\Ext^1_{L_K(E)}(-, \Uni E{c-1})$,  we get
\begin{align*}
\Ext^1_{L_K(E)}(L_K(E)/I, \Uni E{c-1}) &=\Ext^1_{L_K(E)}(\varinjlim L_K(E)/{I_{\lambda}}, \Uni E{c-1})  \hspace{.2in} \\
% \mbox{ (standard)} \\
& = \varprojlim \Ext^1(L_K(E)/{I_{\lambda}}, \Uni E{c-1})   \hspace{.5in}  \mbox{  (by Proposition  \ref{Ucpureinjective})} \\
& = 0  \hspace{2.4in}   \mbox{(by Proposition \ref{prop:div})}. 
\end{align*}
Since $L_K(E)\alpha_1$ is an essential submodule of $\Uni E{c-1}$, the last statement follows.
%By Proposition \ref{prop:div},  $\Ext^1(L_K(E)/{I_{\lambda}}, \Uni E{c-1})=0$ for all $\lambda$.  But then Proposition  \ref{Ucpureinjective} applies  to yield the second equality of the following:
%$$\Ext^1(L_K(E)/I, \Uni E{c-1})=\Ext^1(\varinjlim L_K(E)/{I_{\lambda}}, \Uni E{c-1})=\varprojlim \Ext^1(L_K(E)/{I_{\lambda}}, \Uni E{c-1})=  0.$$
% So we get that $\Uni E{c-1}$ is injective. We conclude, since $L_K(E)\alpha_1$ is an essential submodule of $\Uni E{c-1}$.
% 
% \bigskip
% 
% {\bf A/F:    That $U_c$  is the injective envelope follows from ... ???   }
% 
\end{proof} 

{\it  Proof of Theorem \ref{thm:injective}(1)}.   This now follows immediately from Propositions \ref{reducetosourceloop}, \ref{prop;endomorphism},  and \ref{sourceloopresult}.

\bigskip

\begin{center}
 \textsc{Acknowledgements}
\end{center}

The first  author is partially supported by a Simons Foundation Collaboration Grants for Mathematicians Award \#208941.     The second and third authors are supported by Progetto di Eccellenza Fondazione Cariparo ``Algebraic structures
and their applications: Abelian and derived categories, algebraic entropy and representation of algebras''.    Part of this work was carried out during a visit by the second and third authors to the University of Colorado Colorado Springs, and during a visit by the first author to the Universit\`{a} degli Studi di Padova.    The authors are grateful for the support of these institutions.

\end{document}